\documentclass[envcountsame,envcountsect,notbib]{svmult}
\usepackage{makeidx}         
\usepackage{multicol}        
\usepackage[bottom]{footmisc}
\usepackage{url}
\makeindex
\usepackage{amscd,xy,amssymb,amsmath,nicefrac}
\xyoption{all}
\newdir{ >}{{}*!/-5pt/@{>}}
\newcommand{\widebar}{\overline}
\newcommand{\diag}{\mathrm{diag}}
\newcommand{\mbar}{\mathrm{bar}}
\newcommand{\ev}{{\rm ev}}

\newcommand{\ca}{\mathfrak{C}}

\newcommand{\com}{\mathfrak{K}}

\newcommand{\fk}{\com}
\newcommand{\A}{\mathbb{A}}
\newcommand{\fD}{\mathfrak{D}}
\newcommand{\comment}[1]{}
\newcommand{\B}{\mathcal{B}}

\newcommand{\cB}{\B}
\renewcommand{\I}{\mathcal{I}}
\newcommand{\J}{\mathcal{J}}
\newcommand{\cK}{\mathcal{K}}

\newcommand{\ba}{\mathfrak{B}}

\newcommand{\cC}{\mathcal{C}}
\renewcommand{\E}{\mathcal{E}}
\newcommand{\F}{\mathcal{F}}
\newcommand{\cF}{\F}

\newcommand{\M}{\mathcal{M}}
\newcommand{\cL}{\mathcal{L}}
\newcommand{\cT}{\mathcal{T}}

\def\idem{\mathrm{Idem}}
\def\GL{\mathrm{GL}}
\renewcommand{\inf}{\rm inf}

\newcommand{\dif}{\rm dif}
\newcommand{\nil}{\rm nil}
\renewcommand{\top}{\rm top}

\newcommand{\sm}{\rm sm}

\newcommand{\iso}{\overset{\cong}{\to}}
\newcommand{\weq}{\overset{\sim}{\to}}

\newcommand{\fib}{\twoheadrightarrow}

\newcommand{\DA}{\Delta}
\newcommand{\DD}{\Delta^{\dif}}
\newcommand{\triqui}{\triangleleft}

\newcommand{\ass}{\mathfrak{Ass}}
\newcommand{\ab}{\mathfrak{Ab}}

\newcommand{\Grp}{\mathfrak{Grp}}
\newcommand{\fC}{\mathfrak{C}}
\newcommand{\lc}{{\rm lc}}

\newcommand{\ho}{\mathrm{Ho}}
\newcommand{\topo}{\mathrm{Top}}
\newcommand{\joto}{\ho\topo}

\newcommand{\tor}{\mathrm{Tor}}
\newcommand{\coker}{\mathrm{coker}}
\newcommand*{\colim}{\mathop{\mathrm{colim}}}
\newcommand{\coeq}{\mathrm{coeq}}

\newcommand{\hofi}{\mathop{\mathrm{hofiber}}}
\newcommand{\ima}{\mathrm{Im}}
\newcommand{\hotimes}{\hat{\otimes}}
\newcommand{\sotimes}{\overset{\sim}\otimes}

\newcommand{\End}{\mathrm{End}}
\newcommand{\N}{\mathbb{N}}
\newcommand{\Z}{\mathbb{Z}}
\newcommand{\Q}{\mathbb{Q}}
\newcommand{\R}{\mathbb{R}}
\newcommand{\C}{\mathbb{C}}
\newcommand{\K}{\mathbb{K}}
\newcommand{\KH}{\mathbb{KH}}
\newcommand{\KD}{\mathbb{KD}}
\newcommand{\HN}{\mathbb{HN}}

\newcommand{\HP}{\mathbb{HP}}
\newcommand{\HC}{\mathbb{HC}}

\newcommand{\sn}{\par\smallskip\noindent}
\newcommand{\mn}{\par\medskip\noindent}
\newcommand{\bn}{\par\bigskip\noindent}

\spnewtheorem{them}{Theorem}[subsection]{\bf}{\it}
\spnewtheorem{properties}[them]{Properties}{\bf}{\it}
\spnewtheorem*{notation}{Notation}{\it}{\rm}
\spnewtheorem{propo}[them]{Propostion}{\bf}{\it}
\spnewtheorem{prop}{Propostion}[section]{\bf}{\it}
\spnewtheorem{lem}[them]{Lemma}{\bf}{\it}
\spnewtheorem{coro}[them]{Corollary}{\bf}{\it}
\spnewtheorem{exe}[them]{Exercise}{\it}{\rm}
\spnewtheorem{exa}[them]{Example}{\bf}{\rm}
\spnewtheorem{examples}[them]{Examples}{\bf}{\rm}
\spnewtheorem{exam}[prop]{Example}{\bf}{\rm}
\spnewtheorem{defi}[them]{Definition}{\bf}{\it}
\spnewtheorem{rem}[them]{Remark}{\it}{\rm}

\begin{document}
\title*{Algebraic v. topological $K$-theory: a friendly match}
\author{Guillermo Corti\~nas}
\institute{Dep. Matem\'atica\\ Ciudad Universitaria Pab 1\\ (1428) Buenos Aires, Argentina\\ \texttt{gcorti@dm.uba.ar} }

\thanks{Work for this notes was partly supported by FSE and by grants PICT03-12330, UBACyT-X294, VA091A05, and
MTM00958.}

\maketitle


\section{Introduction}

These notes evolved from the lecture notes of a minicourse given in Swisk, the Sedano Winter School on K-theory held in Sedano, Spain, during the week January 22--27 of 2007, and from
those of a longer course given in the University of Buenos Aires, during
the second half of 2006. They intend to be an introduction to $K$-theory, with emphasis in the comparison between its algebraic and topological variants.
We have tried to keep as elementary as possible.
Section \ref{sec:knle1} introduces $K_n$ for
$n\le 1$. Elementary properties such as matrix stability and excision are discussed. Section \ref{sec:topk}
is concerned with topological $K$-theory of Banach algebras; its excision property is derived from the excision sequence
for algebraic $K_0$ and $K_1$. Cuntz' proof of Bott periodicity for $C^*$-algebras,
via the $C^*$-Toeplitz extension, is sketched. In the next section we review Karoubi-Villamayor $K$-theory,
which is an algebraic version of $K^{\top}$, and has some formally similar properties, such as (algebraic)
homotopy invariance,  but does not satisfy
excision in general. Section \ref{sec:kh} discusses $KH$, Weibel's homotopy $K$-theory, which is introduced
in a purely algebraic, spectrum-free manner. Several of its properties, including excision, homotopy
invariance and the fundamental theorem, are proved. The parallelism between Bott periodicity and the fundamental theorem for
$KH$ is emphasized by the use of the algebraic Toeplitz extension in the proof of the latter.
Quillen's higher $K$-theory is introduced in Section \ref{sec:kq}, via the plus construction of the classifying
space of the general linear group. This is the first place where some algebraic topology is needed.
The ``d\'ecalage" formula $K_n\Sigma R=K_{n-1}R$ via Karoubi's suspension is proved, and some
some of the deep results of Suslin and Wodzicki on excision are discussed. Then the fundamental theorem for
$K$-theory is reviewed, and its formal connection to Bott periodicity via the algebraic Toeplitz extension is
established. The next section is the first of three devoted to the comparison between algebraic and topological
$K$-theory of topological algebras. Using Higson's homotopy invariance theorem, and the excision results of Suslin
and Wodzicki, we give proofs of
the $C^*$- and Banach variants of Karoubi's conjecture, that algebraic and topological $K$-theory become
isomorphic after stabilizing with respect to the ideal of compact operators (theorems of Suslin-Wodzicki and Wodzicki, respectively).
Section \ref{sec:topk2} defines two variants of topological $K$-theory for locally convex algebras:
$KV^{\dif}$ and $KD$ which are formally analogue to $KV$ and $KH$. Some of their basic properties are similar
and derived with essentially the same arguments as their algebraic counterparts. We also give a proof
of Bott periodicity for $KD$ of locally convex algebras stabilized by the algebra of smooth compact
operators. The proof uses the locally convex Toeplitz extension, and is modelled on Cuntz' proof of Bott periodicity
for his bivariant $K$-theory of
locally convex algebras. In Section \ref{sec:compa2} we review some of the results of \cite{cot}. Using
the homotopy invariance theorem of Cuntz and Thom, we show that $KH$ and $KD$ agree on locally
convex algebras stabilized by Fr\'echet operator ideals. The spectra for Quillen's and Weibel's $K$-theory, and the space for Karoubi-Villamayor
$K$-theory are introduced in Section \ref{sec:spectra}, where also the primary and secondary characters going
from $K$-theory to cyclic homology are reviewed. The technical results of this section are used in the
next, where we again deal with the comparison between algebraic and topological $K$-theory of locally
convex algebras. We give proofs of the Fr\'echet variant of Karoubi's conjecture (due to Wodzicki),
and of the $6$-term exact sequence of \cite{cot}, which relates algebraic $K$-theory and cyclic homology to topological $K$-theory
of a stable locally convex algebra.
\section{The groups $K_n$ for $n\le 1$.}\label{sec:knle1}

\begin{notation} Throughout these notes, $A, B, C$ will be rings and $R, S, T$ will be rings with unit.
\end{notation}
\subsection{Definition and basic properties of $K_j$ for $j=0,1$.}\label{sec:basick01}
\bn
Let $R$ be a ring with unit. Write $M_nR$ for the matrix ring. Regard $M_nR\subset M_{n+1}R$ via
\begin{equation}\label{inclumat}
a\mapsto \left[\begin{array}{cc} a & 0\\ 0 & 0\end{array}\right]
\end{equation}
Put
\[
M_\infty R=\bigcup_{n=1}^\infty M_nR
\]
Note $M_\infty R$ is a ring (without unit). We write $\idem_n R$ and $\idem_\infty R$ for the set of idempotent elements of $M_nR$
and $M_\infty R$.
Thus
\[
M_\infty R\supset\idem_\infty R=\bigcup_{n=1}^\infty \idem_nR.
\]
We write $\GL_nR=(M_nR)^*$ for the group of invertible matrices. Regard $\GL_nR\subset \GL_{n+1}R$ via
\[
g\mapsto\left[\begin{array}{cc}g&0\\0&1\end{array}\right]
\]
Put
\[
\GL R:=\bigcup_{n=1}^\infty\GL_nR.
\]
Note $\GL R$ acts by conjugation on $M_\infty R$, $\idem_\infty R$ and, of course, $\GL R$.

For $a,b\in M_\infty R$ there is defined a {\it direct sum} operation
\begin{equation}\label{directsum}
a\oplus b:=\left[\begin{array}{ccccccc}a_{1,1}&0&a_{1,2}&0&a_{1,3}&0&\dots\\
                               0 & b_{1,1}&0&b_{1,2}&0&b_{1,3}&\dots\\
                               a_{2,1}&0&a_{2,2}&0&a_{2,3}&0&\dots\\
                                 \vdots &\vdots &\vdots &\vdots &\vdots & \vdots& \ddots\end{array}\right].
\end{equation}
We remark that if $a, b\in M_pR$ then $a\oplus b\in M_{2p}R$ and is conjugate, by a permutation matrix, to the
usual direct sum
\[
\left[\begin{array}{cc}a&0\\0&b\end{array}\right].
\]
One checks that $\oplus$ is associative and commutative up to conjugation. Thus the coinvariants under the conjugation action
\[
I(R):=((\idem_\infty R)_{\GL R},\oplus)
\]
form an abelian monoid.

\begin{exe}\label{exe:othersums}
The operation \eqref{directsum} can be described as follows. Consider the decomposition $\N=\N_0\coprod \N_1$
into even and odd positive integers; write $\phi_i$ for the bijection $\phi_i:\N\to \N_i$, $\phi_i(n)=2n-i$ $i=0,1$.
The map $\phi_i$ induces an $R$-module monomorphism
\[
\phi_i:R^{(\N)}:=\bigoplus_{n=1}^\infty R\to R^{(\N_i)}\subset R^{(\N)},\qquad e_n\mapsto e_{\phi_i(n)}.
\]
We abuse notation and also write $\phi_i$ for the matrix of this homomorphism with respect to the canonical basis and $\phi_i^t$ for its transpose. Check the formula
\[
a\oplus b=\phi_0a\phi^t_0+\phi_1a\phi^t_1.
\]
Observe that the same procedure can be applied to any decomposition $\N=\N'_0\coprod \N'_1$ into two infinite disjoint subsets
and any choice of bijections $\phi'_i:\N\to \N'_i$, to obtain
an operation $\oplus_{\phi'}:M_\infty R\times M_\infty R\to M_\infty R$. Verify that the operation so obtained defines
the same monoid structure on the coinvariants $(M_\infty R)_{\GL R}$, and thus also on $I(R)$.
\end{exe}

\begin{lem} Let $M$ be an abelian monoid. Then there exist an abelian group $M^+$ and a monoid homomorphism $M\to M^+$
such that if $M\to G$ is any other such homomorphism, then there exists a unique group homomorphism $M^+\to G$ such that
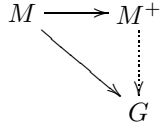
\[
\xymatrix{M\ar[r]\ar[dr]&M^+\ar@{.>}[d]\\&G}
\]
commutes.
\end{lem}
\begin{proof} Let $F=\Z^{(M)}$ be the free abelian group on one generator $e_m$ for each $m\in M$, and let $S\subset F$ be the subgroup generated by all elements of
the form $e_{m_1}+e_{m_2}-e_{m_1+m_2}$. One checks that $M^+=F/S$ satisfies the desired properties.\qed
\end{proof}

\begin{defi}\label{defi:k01}
\begin{gather*}
K_0(R):=I(R)^+\\
K_1(R):=\frac{\GL R}{[\GL R,\GL R]}=(\GL R)_{ab}.
\end{gather*}
Here $[,]$ denotes the commutator subgroup, and the subscript ${}_{ab}$ indicates abelianization.
\end{defi}
\sn
\sn
\begin{propo}(see \cite[Section 2.1]{rosen})
\begin{itemize}
\item $[\GL R,\GL R]=ER:=<1+ae_{i,j}:a\in R, i\ne j>$, the subgroup of $\GL R$ generated by elementary matrices.
\item If $\alpha\in \GL_n R$ then
\[
\left[\begin{array}{cc}\alpha &0\\0&\alpha^{-1}\end{array}\right]\in E_{2n}R\text{ \qquad\qquad (Whitehead's Lemma).}
\]
(Here $E_{2n}R=ER\cap\GL_{2n}R$).\qed
\end{itemize}
\end{propo}
As a consequence of Whitehead's lemma above, if $\beta\in \GL_n R$, then
\begin{align}\label{whitehead}
\alpha\beta=&\left[\begin{array}{ccc}\alpha\beta&0\\0&1_{n\times n}\end{array}\right]\nonumber\\
=&\left[\begin{array}{cc}\alpha&0\\0&\beta\end{array}\right]\left[\begin{array}{cc}\beta &0\\0&\beta^{-1}\end{array}\right]\\
\equiv&\alpha\oplus\beta\quad\mod ER.\nonumber
\end{align}
\begin{exe}\label{exe:othersums2}
Let $R$ be a unital ring, and let $\phi'$ and $\oplus_{\phi'}$ be as in Exercise \ref{exe:othersums}.
Prove that $\oplus_{\phi'}$ and $\oplus$ define the same operation in $K_1(R)$, which coincides with the product
of matrices.
\end{exe}
\mn
Let $r\ge 1$. Then
\[
p_r=1_{r\times r}\in\idem_\infty R.
\]
Because $p_r\oplus p_s=p_{r+s}$, the assignment $r\mapsto p_r$ defines a monoid homomorphism $\N\to I(R)$. Applying the group completion functor
we obtain a group homomorphism
\begin{equation}\label{ztok0}
\Z=\N^+\to I(R)^+=K_0R.
\end{equation}
Similarly, the inclusion $R^*=\GL_1R\subset \GL R$ induces a homomorphism
\begin{equation}\label{antidet}
R^*_{ab}\to K_1R.
\end{equation}
\begin{exa}\label{exa:conmut}
If $F$ is a field, and $e\in \idem_\infty F$ is of rank $r$, then $e$ is conjugate to $p_r$; moreover $p_r$ and $p_s$ are conjugate
$\iff$ $r=s$. Thus \eqref{ztok0} is an isomorphism in this
case. Assume more generally that $R$ is commutative. Then \eqref{ztok0} is a split monomorphism. Indeed,
there exists a surjective unital homomorphism $R\fib F$ onto a field $F$; the induced map $K_0(R)\to K_0(F)=\Z$ is a left inverse
of \eqref{ztok0}. Similarly, for commutative $R$, the homomorphism \eqref{antidet} is injective, since it is split by the map $det:K_1R\to R^*$ induced by the determinant.
\end{exa}
\begin{exa}\label{dualn} The following are examples of rings for which the maps \eqref{ztok0} and \eqref{antidet} are isomorphisms
(see \cite[Ch.1\S3, Ch.2\S2,\S3]{rosen}): fields, division rings, principal ideal domains and local rings. Recall that a ring $R$ is a {\it local ring} if the
subset $R\backslash R^*$ of noninvertible elements
is an ideal of $R$. For instance if $k$ is a field, then the $k$-algebra $k[\epsilon]:=k\oplus k\epsilon$ with $\epsilon^2=0$ is
a local ring. Indeed $k[\epsilon]^*=k^*+k\epsilon$ and $k[\epsilon]\backslash k[\epsilon]^*=k\epsilon\triqui k[\epsilon]$.
\end{exa}
\begin{exa}\label{exa:rsubi} Here is an example of a local ring involving operator theory. Let $H$ be a separable Hilbert space over $\C$; put
$\cB=\cB(H)$ for the algebra of bounded operators. Write $\cK\subset \B$ for the ideal
of compact operators, and $\cF$ for that of finite rank operators. The Riesz-Schauder theorem from elementary operator theory
implies that if $\lambda\in\C^*$ and $T\in \cK$ then there exists an $f\in\cF$ such that $\lambda+T+f$ is invertible in $\B$.
In fact one checks that if $\cF\subset I\subset\cK$ is an ideal of $\cB$ such that $T\in I$ then the inverse of $\lambda+T+f$
is again in $\C\oplus I$. Hence the ring
\[
R_I:=\C\oplus I/\cF
\]
is local, and thus $K_0(R_I)=\Z$.
\end{exa}
\begin{rem}\label{rem:k0projmod}({\it $K_0$ from projective modules})
In the literature, $K_0$ of a unital ring is often defined in terms of finitely generated projective modules. This approach is equivalent to
ours, as we shall see presently.
If $R$ is a unital ring and $e\in \idem_nR$, then left multiplication by $e$ defines a right module homomorphism $R^n=R^{n\times 1}\to R^n$ with
image $eR^n$. Similarly $(1-e)R^n\subset R^n$ is a submodule, and we have a direct sum decomposition
\[
R^n=eR^n\oplus (1-e)R^n.
\]
Hence $eR^n$ is a  finitely generated projective module, as it is a direct summand of a finitely generated free $R$-module. Note every finitely
generated projective right $R$-module arises in this way for some $n$ and some $e\in\idem_nR$. Moreover, one checks that if $e\in \idem_nR$ and $f\in\idem_mR$, then the modules $eR^n$ and $fR^m$ are isomorphic if and only if the images of $e$ and $f$ in $\idem_\infty R$ define the same class in $I(R)$ (see \cite[Lemma 1.2.1]{rosen}). Thus we have a natural bijection from the monoid $I(R)$ to the set $P(R)$ of isomorphism classes of finitely generated projective modules; further, one checks that the direct sum of idempotents corresponds to the direct sum of modules. Hence the monoids $I(R)$ and $P(R)$ are isomorphic, and therefore they have the same group completion:
\[
K_0(R)=I(R)^+=P(R)^+.
\]
\end{rem}

\sn

\paragraph{\it Additivity.}
 If $R_1$ and $R_2$ are unital rings, then $M_\infty(R_1\times R_2)\to M_\infty R_1\times M_\infty R_2$ is an
isomorphism. It follows from this that the natural map induced by the projections $R_1\times R_2\to R_i$ is an isomorphism:
\[
K_j(R_1\times R_2)\to K_jR_1\oplus K_jR_2\qquad (j=0,1).
\]
\sn
\paragraph{\it Application: extension to nonunital rings.}
If $A$ is any (not necessarily unital) ring, then the abelian group $\tilde{A}=A\oplus\Z$ equipped with the multiplication
\begin{equation}\label{unital_prod}
(a+n)(b+m):=ab+nm\qquad (a,b\in A, \ \ n,m\in\Z)
\end{equation}
is a unital ring, with unit element $1\in\Z$, and $\tilde{A}\to\Z$, $a+n\mapsto n$, is a unital homomorphism. Put
\[
K_j(A):=\ker(K_j\tilde{A}\to K_j\Z) \qquad (j=0,1).
\]
If $A$ happens to have a unit, we have two definitions for $K_jA$. To check that they are the same, one observes that the map
\begin{equation}\label{unitalize}
\tilde{A}\to A\times\Z, \ \ a+n\mapsto (a+n\cdot 1,n)
\end{equation}
is a unital isomorphism. One verifies that, under this isomorphism, $\tilde{A}\to \Z$ identifies with the projection
$A\times \Z\to \Z$, and
$\ker(K_j(\tilde{A})\to K_j\Z)$ with $\ker(K_jA\oplus K_j\Z\to K_j\Z)=K_jA$. Note that the same procedure works to extend
any additive functor of unital rings unambiguously to all rings.
\sn
\begin{notation} We write $\ass$ for the category of rings and ring homomorphisms, and $\ass_1$ for the subcategory of
unital rings and unit preserving ring homomorphisms.\end{notation}
\sn
\begin{rem}\label{rem:nunitalk1}
The functor $\GL:\ass_1\to\Grp$ preserves products. Hence it extends to all rings by
\[
\GL(A):=\ker(\GL(\tilde{A})\to \GL\Z)
\]
It is a straightforward exercise to show that, with this definition, $\GL$ becomes a left exact functor in $\ass$; thus if $A\triqui B$ is
an ideal embedding, then $\GL(A)=\ker(\GL(B)\to\GL(B/A))$. It is straightforward from this that the group $K_1A$ defined above can be described as
\begin{equation}\label{form=k1nunital}
K_1A=\GL(A)/E(\tilde{A})\cap \GL(A)
\end{equation}
A little more work shows that $E(\tilde{A})\cap \GL(A)$ is the smallest normal subgroup of $E(\tilde{A})$ which
contains the elementary
matrices $1+a e_{i,j}$ with $a\in A$ (see \cite[2.5]{rosen}).
\end{rem}
\sn
\paragraph{\it Matrix stability.} Let $R$ be a unital ring and $n\ge 2$.
A choice of bijection $\phi:\N\times \N_{\le n}\cong\N$
gives a ring isomorphism 
$\phi:M_\infty(M_nR)\cong M_\infty(R)$ which induces, for $j=0,1$, a group isomorphism
$\phi_j:K_j(M_nR)\cong K_jR $.
Next, consider the decomposition $\N=\N'_0\coprod\N'_1$, $\N'_0=\phi(\N\times\{1\})$,
 $\N'_1=\phi(\N\times \N_{<n}\backslash \N\times\{1\})$. Setting $\psi_0:\N\to \N'_0$,
$\psi_0(m)=\phi(m,1)$ and choosing any bijection $\psi_1:\N\to \N'_1$, we obtain, as in Exercise \ref{exe:othersums},
 a direct sum operation $\oplus_\psi:M_\infty R\times M_\infty R\to M_\infty R$. Set $\iota: R\mapsto M_nR$, $r\mapsto re_{11}$. The composite
of $M_\infty\iota$ followed by the isomorphism induced by $\phi$ is the map sending
\begin{equation}\label{map:+0}
e_{i,j}(r)\mapsto e_{\phi(i,1),\phi(j,1)}(r)=e_{i,j}(r)\oplus_\psi 0.
\end{equation}
By 
Exercise \ref{exe:othersums} the latter map induces the identity in $K_0$. Moreover, one checks that \eqref{map:+0} induces the map $\GL(M_nR)\to \GL(M_nR)$, $g\mapsto g\oplus_\psi 1$, whence it also gives the identity in $K_1$, by Exercise \ref{exe:othersums2}. It follows that, for $j=0,1$, the map
\[
K_j(\iota):K_j(R)\to K_j(M_nR)
\]
is an isomorphism, inverse to $\phi_j$. Starting with a bijection $\phi:\N\times \N\to \N$ and using the same argument as above, one
shows that also
\[
K_j(\iota):K_j(R)\to K_j(M_\infty R)
\]
is an isomorphism.

\paragraph{\it Nilinvariance for $K_0$.} If $I\triqui R$ is a nilpotent ideal, then $K_0(R)\to K_0(R/I)$ is an isomorphism.
This property is a consequence of the well-known fact that nilpotent extensions admit idempotent liftings, and that any
two liftings of the same idempotent are conjugate (see for example \cite[1.7.3]{ben}). Note that $K_1$ does not have the same property, as the following
example shows.
\sn
\begin{exa}\label{exa:dualn2}
Let $k$ be a field. Then by \ref{dualn}, $K_1(k[\epsilon])=k^*+k\epsilon$ and $K_1(k)=k^*$. Thus $k[\epsilon]\to k[\epsilon]/\epsilon k[\epsilon]=k$
does not become an isomorphism under $K_1$.
\end{exa}

\begin{exa}\label{exa:square-zero}
Let $A$ be an abelian group; make it into a ring with the trivial product: $ab=0$ $\forall a,b\in A$.
The map $A\to\GL_1A$, $a\mapsto 1+a$ is an isomorphism of groups, and thus induces a group homomorphism
$A\to K_1A$. We are going to show that the latter map is an isomorphism. First of all, it is injective,
since $\GL_1(\tilde{A)}\to K_1(\tilde{A})$ is (by \ref{exa:conmut}) and since by definition,
$K_1A\subset K_1(\tilde{A})$. Second, note that if $\epsilon=1+ae_{ij}$ is an elementary matrix with $a\in A$
and $g\in\GL A$, then $(\epsilon g)_{ij}=g_{ij}+a$, and $(\epsilon g)_{p,q}=g_{p,q}$ for $(p,q)\ne (i,j)$.
Thus $g$ is congruent to its diagonal in $K_1 A$. But by Whitehead's lemma, any diagonal matrix in $\GL(\tilde{A})$
is $K_1$-equivalent to its determinant (see \eqref{whitehead}). This shows that $A\to K_1A$ is surjective, whence an isomorphism.
\end{exa}
\begin{rem}
The example above shows that $K_1$ is no longer matrix stable when extended to general nonunital rings. In addition,
it gives another example of the failure of nilinvariance for $K_1$ of unital rings. It follows from
\ref{exa:dualn2} and \ref{exa:square-zero} that if $k$ and $\epsilon$ are as in Example \ref{exa:dualn2}, then
$K_1(k\epsilon)=\ker(K_1(k[\epsilon])\to K_1(k))$. In \ref{exa:swanex}
below, we give an example of a unital ring $T$ such that $k\epsilon$ is an ideal in $T$, and
such that $\ker(T\to T/k\epsilon)=0$.
\end{rem}
\begin{exe}
Prove that $K_0$ and $K_1$ commute with filtering colimits; that is, show that if $I$ is a small filtering
category and $A:I\to \ass$ is a functor, then for $j=0,1$, the map $\colim_I K_jA_i\to K_j(\colim_IA_i)$ is an
isomorphism.
\end{exe}
\subsection{Matrix-stable functors.}

\begin{defi}\label{defi:stability} Let $\fC\subset\ass$ be a subcategory of the category of rings, $S:\fC\to \fC$ a functor,
and $\gamma:1_\fC\to S$ a natural transformation. If $\fD$ is any category, $F:\fC\to\fD$ a functor and $A\in\fC$,
then we say that $F$ is
{\rm stable on $A$} with respect to $(S,\gamma)$ (or $S$-stable on $A$, for short) if the map $F(\gamma_A):F(A)\to F(S(A))$
is an isomorphism. We say that $F$ is {\rm $S$-stable} if it is stable on every $A\in\fC$.
\end{defi}
\begin{exa}
We showed in Section \ref{sec:basick01} that $K_j$ is $M_n$ and even $M_\infty$-stable on unital rings; in both cases, the natural transformation of the definition
above is $r\mapsto re_{11}$.
\end{exa}

\begin{exe}\label{exe:matrix} Let $F:\ass\to \ab$ be a functor and $A$ a ring. Prove:
\item{i)} The following are equivalent:
\begin{itemize}
\item For all $n,p\in \N$, $F$ is $M_p$-stable on $M_nA$.
\item For all $n\in \N$, $F$ is $M_2$-stable on $M_nA$
\end{itemize}
In particular, an $M_2$-stable functor is $M_n$-stable, for all $n$.
\item{ii)} If $F$ is $M_\infty$-stable on both $A$ and $M_nA$, then $F$ is $M_n$-stable on $A$. In particular,
if $F$ is $M_\infty$-stable, then it is $M_n$-stable for all $n$.
\end{exe}

\begin{lem}\label{lem:i_0=i_1}
Let $F:\ass\to \fD$ be a functor, and $A\in\ass$. Assume $F$ is $M_2$-stable on both $A$ and $M_2A$. Then the inclusions
$\iota_0,\iota_1:A\to M_2A$
\[
 \iota_0(a)=ae_{11},\qquad \iota_1(a)=ae_{22}
\]
induce the same isomorphism $FA\to FM_2A$.
\end{lem}
\begin{proof}
Consider the composites $j_0=\iota_0M_2\circ\iota_0$ and $j_1=\iota_0M_2\circ\iota_1$, and the matrices
\[
J_2=\left[\begin{array}{cccc} 0&1&0&0\\ 1&0&0&0\\ 0&0&1&0\\ 0&0&0&1\end{array}\right],\quad
J_3=\left[\begin{array}{cccc} 0&0&1&0\\ 1&0&0&0\\ 0&1&0&0\\ 0&0&0&1 \end{array}\right]\in \GL_4\Z.
\]
Conjugation by $J_i$ induces an automorphism $\sigma_i$ of $M_4A=M_2M_2A$ of order $i$ such that
\[
\sigma_ij_0=j_1\qquad (i=2,3).
\]
Since $F(j_0)$ is an isomorphism, and the orders of $\sigma_2$ and $\sigma_3$ are relatively prime, it follows that $F(\sigma_2)=F(\sigma_3)=1_{F(M_4A)}$ and hence that
$F(j_0)=F(j_1)$ and $F(\iota_0)=F(\iota_1)$.\qed
\end{proof}
\begin{exe}\label{exe:add} Let $F$ and $A$ be as in Lemma \ref{lem:i_0=i_1}. Assume in addition that $\fD$ and $F$ are
additive. Consider the map
\[
\diag:A\times A\to M_2A,\ \ \diag(a,b)=\left[\begin{array}{cc}a&0\\ 0&b\end{array}\right].
\]
Prove that the composite
\[
F(A)\oplus F(A)=F(A\times A)\overset{F(\diag)}\longrightarrow F(M_2A)\overset{F(i_0)^{-1}}\longrightarrow F(A)
\]
is the codiagonal map (i.e. it restricts to the identity on each copy of $F(A)$).
\end{exe}

\begin{propo}\label{prop:vw} Let $F$ and $A$ be as in Lemma \ref{lem:i_0=i_1}, $A\subset B$ an overring, and
$V,W\in B$ elements such that
\[
WA, AV\subset A,\ \ aVWa'=aa'\ \ (a,a'\in A).
\]
Then
\[
\phi^{V,W}:A\to A,\quad a \mapsto WaV
\]
is a ring homomorphism, and
\[
F(\phi^{V,W}) = 1_{F(A)}.
\]
\end{propo}
\begin{proof} We may assume that $B$ is unital. Consider the elements $V\oplus 1$ and $W\oplus 1\in M_2B$. The hypothesis
guarantee that both $\phi:=\phi^{V,W}$ and $\phi':=\phi^{V \oplus 1,W \oplus 1}:M_2A\to M_2A$ are well-defined ring homomorphisms. Moreover, $\phi'\iota_1=\iota_1$
and $\phi'\iota_0=\iota_0\phi$. It follows that $F(\phi')$ and $F(\phi)$ are the identity
maps, by Lemma \ref{lem:i_0=i_1}.\qed
\end{proof}
\begin{exe}\label{exe:finrank}
\item{i)} Let $R$ be a unital ring and $L$ a free, finitely generated $R$-module of rank $n$. A choice
of basis $\ba$ of $L$ gives an isomorphism $\phi=\phi_\ba:M_nR\to \End_RL$. Use \ref{prop:vw}
to show that $K_j(\phi)$ is independent of the choice of $\ba$ $(j=0,1)$.
\sn
\item{ii)} Assume $R$ is a field. If $e\in \End_RL$ is idempotent, then $\iota_e:R\to\End_RL$, $x\mapsto xe$
is a ring monomorphism. Show that if $e\in \End_RL$ is of rank $1$, then $K_j(\iota_e)=K_j(\phi\iota)$.
In particular, $K_j(\iota_e)$ is independent of the choice of the rank-one idempotent $e$.

\item{iii)} Let $H$ and $\cF$ be as in Example \ref{exa:rsubi}. If $V\subset W\subset H$ are finite dimensional
subspaces and $U=V^{\perp}\cap W$ then the decomposition $W=V\oplus U$ induces an inclusion
$\End_\C(V)\subset \End_\C(W)$. Show that
\[
\cF=\bigcup_{\dim V<\infty}\End_\C(V)
\]
\item{iv)} Prove that if $e\in\cF$ is any self-adjoint, rank-one idempotent, then the inclusion $\C\to\cF$, $x\mapsto xe$,
induces an isomorphism $K_j(\C)\iso K_j(\cF)$. Show moreover that this isomorphism is independent of the
choice of $e$.
\end{exe}

\subsection{Sum rings and infinite sum rings.}
Recall from \cite{wa} that
a {\it sum ring} is a unital ring $R$ together with elements $\alpha_i,\beta_i$, $i=0,1$ such that the following
identities hold
\begin{gather}
\alpha_0\beta_0=\alpha_1\beta_1=1\nonumber\\
\beta_0\alpha_0+\beta_1\alpha_1=1\label{sumring}
\end{gather}
If $R$ is a sum ring, then
\begin{gather}\label{boxplus}
\boxplus:R\times R\to R,\\
(a,b)\mapsto a\boxplus b=\beta_0a\alpha_0+\beta_1b\alpha_1\nonumber
\end{gather}
is a unital ring homomorphism.
\comment{In fact Wagoner has shown that if $a,b\in R$ then there is a matrix $Q\in \GL_3R$ which
conjugates $a\boxplus b$ to $a\oplus b$ (\cite[page 355]{wa}). But as $K_n$ is a matrix stable functor of unital
rings, inner automorphisms induce the identity on $K_nR$ (\cite[5.1.2]{biva}), whence
$\oplus$ and $\boxplus$ are the same operation on $K_nR$. Thus $\boxplus$ is the sum in $K_nR$.}
An {\it infinite sum ring} is a sum ring $R$ together with a unit preserving
ring homomorphism $\infty:R\to R$, $a\mapsto a^\infty$ such that
\begin{equation}\label{ainfi}
a\boxplus a^\infty=a^\infty\qquad (a\in R).
\end{equation}

\begin{propo}\label{prop:sumring}
Let $\fD$ be an additive category, $F:\ass\to\fD$ a functor, and $R$ a sum ring. Assume that the sum
of projection maps $\gamma=F\pi_0+F\pi_1:F(R\times R)\to FR\oplus FR$ is an isomorphism, and that $F$
is $M_2$-stable on both $R$ and $M_2R$.
Then the composite
\[
\xymatrix{F(R)\oplus F(R)\ar[r]^(0.6){\gamma^{-1}} &F(R\times R)\ar[rr]^{F(\boxplus)}&&F(R)}
\]
is the codiagonal map; that is, it restricts to the identity on each copy of $F(R)$. If moreover $R$ is an infinite sum ring,
then $F(R)=0$.
\end{propo}
\begin{proof}
Let $j_0,j_1:R\to R\times R$, $j_0(x)=(x,0)$, $j_1(x)=(0,x)$. Note that $\gamma^{-1}=Fj_0+Fj_1$.
Because $F$ is $M_2$-stable on both $R$ and $M_2R$,
$F(\boxplus)F(j_i)=1_{F(R)}$, by Proposition \ref{prop:vw}. Thus $F(\boxplus)\circ\gamma^{-1}$ is the codiagonal map, as
claimed. It follows that if $\alpha,\beta:R\to R$ are homomorphisms, then $F\alpha+F\beta=F(\boxplus(\alpha,\beta))$.
In particular, if $R$ is an infinite sum ring, then
\[
F(\infty)+1_{F(R)}=F(\infty)+F(1_R)=F(\boxplus(\infty,1_R))=F(\infty).
\]
Thus $1_{F(R)}=0$, whence $F(R)=0$.\qed
\end{proof}

\begin{exa} Let $A$ be a ring. Write $\Gamma A$ for the ring of all $\N\times \N$ matrices
$a=(a_{i,j})_{i,j\ge 1}$ which satisfy the following two conditions:
\begin{itemize}
\item[i)] The set $\{a_{ij}, i,j \in \N \}$ is finite.
\item[ii)] There exists a natural number $N \in \N$ such that each row and each column has at most
$N$ nonzero entries.
\end{itemize}
It is an exercise to show that $\Gamma A$ is indeed a ring and that $M_\infty A\subset\Gamma A$ is an ideal.
The ring $\Gamma A$ is called (Karoubi's) {\it cone ring}; the quotient $\Sigma A:=\Gamma A/M_\infty A$ is the {\it suspension} of $A$.
A useful fact about $\Gamma$ and $\Sigma$ is that the well-known isomorphism
$M_\infty\Z\otimes A\cong M_\infty A$ extends to $\Gamma$, so that there are isomorphisms (see \cite[4.7.1]{biva})
\begin{equation}\label{Gammatenso}
\Gamma \Z\otimes A\iso \Gamma A \text{ and } \Sigma\Z\otimes A\iso \Sigma A.
\end{equation}
\end{exa}
Let $R$ be a unital ring. One checks that
the following elements of $\Gamma R$ satisfy the identities \eqref{sumring}:
\[
\alpha_0 =\sum_{i=1}^\infty e_{i,2i}, \quad \beta_0 =\sum_{i=1}^\infty e_{2i,i},  \quad
\alpha_1 =\sum_{i=1}^\infty e_{i,2i-1}, \quad \mbox{and} \quad \beta_1 =\sum_{i=1}^\infty e_{2i-1,i}.
\]
 Let $a\in\Gamma R$. Because the map $\N\times\N\to \N$, $(k,i)\mapsto 2^{k+1}i+2^k-1$, is injective, the following assignment
gives a well-defined, $\N\times\N$-matrix
\begin{equation}\label{finfty}
\phi^\infty(a) = \sum_{k=0}^\infty \beta_1^k \beta_0 a \alpha_0 \alpha_1^k=\sum_{k,i,j}e_{2^{k+1}i+2^k-1,2^{k+1}j+2^k-1}\otimes a_{i,j}.
\end{equation}
One checks that $\alpha_1\beta_0=\alpha_0\beta_1=0$ and $\alpha_0\alpha_1^i\beta_1^j\beta_0=\delta_{ij}$. It
follows from this that $\phi^\infty$ is a ring endomorphism of $\Gamma R$; it is straightforward that \eqref{ainfi} is satisfied
too. In particular $K_n\Gamma R=0$ for $n=0,1$.

\begin{exe} Let $A$ be a ring. If $m=(m_{i,j})$ is an $\N\times\N$-matrix with coefficients in $A$, and $x\in M_\infty \tilde{A}$,
then both $m\cdot x$ and $x\cdot m$ are well-defined $\N\times \N$-matrices. Put
\begin{equation}\label{waggamma}
\Gamma^\ell A:=\{m\in M_{\N\times\N}A:m\cdot M_\infty \tilde{A}\subset M_\infty A\supset M_\infty \tilde{A}\cdot m\}.
\end{equation}
Prove
\item{i)} $\Gamma^\ell A$ consists of those matrices in $M_{\N\times\N}A$ having
finitely many nonzero elements in each row and column. In particular, $\Gamma^\ell A\supset\Gamma A$.
\item{ii)} The usual matrix sum and product operations make $\Gamma^\ell A$ into a ring.
\item{iii)} If $R$ is a unital ring then $\Gamma^\ell R$ is an infinite sum ring.
\end{exe}
\begin{rem}
The ring $\Gamma^\ell A$ is the cone ring considered by Wagoner in \cite{wa}, where it was denoted $\ell A$. The notion of infinite sum ring
was introduced in {\it loc. cit.}, where it was also shown that if $R$ is unital, then $\Gamma^\ell R$ is an example of such a ring.
\end{rem}
\begin{exe} Let $F:\ass\to\ab$ be a functor. Assume $F$ is both additive and $M_2$-stable for unital rings and for rings of the form $M_\infty R$, with $R$ unital. Show that if $R$ is a unital ring, then the direct sum operation \eqref{directsum}, induces the group operation in $F(M_\infty R)$,
and that the same is true of any of the other direct sum operations of \ref{exe:othersums}.
\end{exe}

\begin{exe} Let $\B$ and $H$ be as in Example \ref{exa:rsubi}. Choose a Hilbert basis $\{e_i\}_{i\ge 1}$ of $H$, and regard
$\B$ as a ring of $\N\times\N$ matrices. With these identifications, show that $\B\supset\Gamma\C$. Deduce from this that $\B$
is a sum ring. Further show that \eqref{finfty} extends to $\B$, so that the latter is in fact an infinite sum ring.
\end{exe}

\subsection{The excision sequence for $K_0$ and $K_1$.}
A reason for considering $K_0$ and $K_1$ as part of the same theory is that they are connected by a long exact
sequence, as shown in Theorem \ref{thm:exci01} below. We need some notation. Let
\begin{equation}\label{abc}
0\to A\to B\to C\to 0
\end{equation}
be an exact sequence of rings. If $\hat{g}\in M_nB$ maps to an invertible matrix $g\in \GL_nC$ and $\hat{g}^*$ maps
to $g^{-1}$, then
\begin{align}\label{elh}
h=h(\hat{g},\hat{g}^*):=&\left[\begin{matrix}2\hat{g}-\hat{g}\hat{g}^*\hat{g}&&\hat{g}\hat{g}^*-1\\1-\hat{g}^*\hat{g}&&\hat{g}^*\end{matrix}\right]\\
                      =& \left[\begin{array}{cc}1&\hat{g}\\ 0&1\end{array}\right]\cdot \left[\begin{array}{cc}1&0\\
-\hat{g}^*&1\end{array}\right]\cdot  \left[\begin{array}{cc}1&\hat{g}\\ 0&1\end{array}\right]\cdot
\left[\begin{array}{cc}0&-1\\ 1&0\end{array}\right]\in \mathrm{E}_{2n}(\tilde{B})\subset\GL_{2n}(\tilde{B})\nonumber
\end{align}
Note that $h$ maps to $\diag(g,g^{-1})\in\GL_{2n}(C)$. Thus $hp_nh^{-1}$ maps to $p_n$, whence
$hp_nh^{-1}-p_n\in M_{2n}A$ and $hp_nh^{-1}\in M_{2n}\tilde{A}$.
Put
\begin{equation}\label{map:partial}
\partial(\hat{g},\hat{g}^*):=[hp_nh^{-1}]-[p_n]\in\ker(K_0(\tilde{A})\to K_0\Z)=K_0A
\end{equation}
\begin{them}\label{thm:exci01}
If \eqref{abc} is an exact sequence of rings, then there is a long exact sequence
\[
\xymatrix{K_1A\ar[r] &K_1B\ar[r]& K_1C\ar[d]^{\partial}\\K_0C&K_0B\ar[l]&K_0A\ar[l]}
\]
The map $\partial$ sends the class of an element $g\in\GL_nC$ to the class of the element \eqref{map:partial}; in
particular the latter depends only on the $K_1$-class of $g$.
\end{them}

\begin{proof}(Sketch) The exactness of the top row of the sequence of the theorem is straightforward.
Putting together \cite[Thms. 1.5.5, 1.5.9 and 2.5.4]{rosen} we obtain the theorem for those sequences \eqref{abc}
in which $B\to C$ is a unital homomorphism. It follows that we have a map of exact sequences
\[
\xymatrix{K_1\tilde{B}\ar[r]\ar[d]&K_1\tilde{C}\ar[d]\ar[r]&
K_0A\ar[d]\ar[r]&K_0\tilde{B}\ar[d]\ar[r]&K_0\tilde{C}\ar[d]\\
K_1\Z\ar@{=}[r]&K_1\Z\ar[r]&0\ar[r]&K_0\Z\ar@{=}[r]&K_0\Z}
\]
Taking kernels of the vertical maps, we obtain an exact sequence
\[
\xymatrix{K_1B\ar[r]&K_1C\ar[r]& K_0A\ar[r]&K_0B\ar[r]&K_0C}
\]
It remains to show that the map $K_1C\to K_0A$ of this sequence is given by the formula of the theorem. This is done
by tracking down the maps and identifications of the proofs of \cite[Thms. 1.5.5, 1.5.9 and 2.5.4]{rosen} (see also
\cite[\S3,\S4]{mil}), and
computing the idempotent matrices to which the projective modules appearing there correspond, taking into account
that $B\to C$ sends the matrix $h\in\GL_{2n}\tilde{B}$ of \eqref{elh} to the diagonal matrix $\diag(g,g^{-1})\in
\GL_{2n}C$.\qed
\end{proof}
\begin{rem}
In \cite[2.5.4]{rosen}, a sequence similar to that of the theorem above is obtained, in which $K_1A$ is replaced
by a relative $K_1$-group $K_1(B:A)$, depending on both $A$ and $B$. For example if $B\to B/A$ is a split
surjection, then (\cite[Exer. 2.5.19]{rosen})
\[
K_1(B:A)=\ker(K_1B\to K_1(B/A))
\]
The groups $K_1(B:A)$ and $K_1A$
are not isomorphic in general (see Example \ref{exa:swanex} below); however their images in $K_1B$ coincide.
We point out also that the theorem above can be deduced directly from Milnor's Mayer-Vietoris sequence for
a Milnor square (\cite[\S4]{mil}).
\end{rem}
The following corollary is immediate from the theorem.

\begin{coro}\label{coro:k0split}
Assume \eqref{abc} is split by a ring homomorphism $C\to B$. Then $K_0A\to K_0B$ is injective, and induces an
isomorphism
\[
K_0A=\ker(K_0B\to K_0C)
\]
Because of this we say that $K_0$ is {\rm split exact}.
\end{coro}
\begin{exa}\label{exa:swanex}(Swan's example \cite{swan})
We shall give an example which shows that $K_1$ is not split exact. Let $k$ be a field
with at least $3$ elements
(i.e. $k\ne\mathbb{F}_2$). Consider the ring of upper triangular matrices
\[
T:=\left[\begin{array}{cc}k&k\\0&k\end{array}\right]
\]
with coefficients in $k$. The set $I$ of strictly upper triangular matrices forms an ideal of $T$, isomorphic as a ring, to the ideal
$k\epsilon\triqui k[\epsilon]$, via the identification $\epsilon=e_{12}$. By Examples \ref{exa:dualn2} and
\ref{exa:square-zero}, $\ker(K_1(k[\epsilon])\to K_1(k))=K_1(k\epsilon)\cong k\epsilon$, the additive group
underlying $k$. If $K_1$ were split exact, then also
\begin{equation}\label{relak1}
K_1(T:I)=\ker(K_1T\to K_1(k\times k))
\end{equation}
should be isomorphic to $k\epsilon$. However we shall see presently that $K_1(T:I)=0$.
Note that $T\to k\times k$ is split by the natural inclusion $\diag:k\times k\to T$. Thus any element of
$K_1(T:I)$ is the class of an element in $\GL(k\epsilon)$, and by \ref{exa:dualn2} it is
congruent to the class of an element in $\GL_1(k\epsilon)=1+k\epsilon$. We shall show that if $\lambda\in k$,
then $1+\lambda\epsilon\in[\GL_1T,\GL_1T]$. Because we are assuming that $k\neq\mathbb{F}_2$, there exists
$\mu\in k-\{0,1\}$; one checks that
\[
1+\lambda \epsilon=\left[\begin{matrix} 1&\lambda\\ 0&1\end{matrix}\right]=\left[\left
[\begin{matrix}\mu&0\\ 0&1\end{matrix}\right],
\left[\begin{matrix}1&\frac{\lambda}{\mu-1}\\ 0&1\end{matrix}\right]\right]\in [\GL_1T,\GL_1T].
\]
\end{exa}
\begin{exa}
Let $R$ be a unital ring. Applying the theorem above to the cone sequence
\begin{equation}\label{coneseq}
0\to M_\infty R\to \Gamma R\to \Sigma R\to 0
\end{equation}
we obtain an isomorphism
\begin{equation}\label{decal}
K_1\Sigma R=K_0R.
\end{equation}
\end{exa}
\begin{exe}\label{exe:k0valid}
Use Corollary \ref{coro:k0split} to prove that all the properties of $K_0$ stated in \ref{sec:basick01} for
unital rings, remain valid for all rings. Further, show that $K_0(\Gamma A)=0$ for all rings $A$, and thus that for any
ring $A$, the boundary map gives a surjection
\[
K_1\Sigma A\fib K_0 A.
\]
\end{exe}

\subsection{Negative $K$-theory.}

\begin{defi}\label{defi:kneg} Let $A$ be a ring and $n\ge 0$. Put
\[
K_{-n}A:=K_0\Sigma^nA.
\]
\end{defi}

\mn
\begin{propo}\label{prop:kneg}
\item{i)} For $n\le 0$, the functors $K_n:\ass\to \ab$ are additive, nilinvariant
and $M_\infty$-stable.
\item{ii)} The exact sequence of \ref{thm:exci01} extends to negative $K$-theory. Thus if
\[
0\to A\to B\to C\to 0
\]
is a short exact sequence of rings, then for $n\le 0$ we
have a long exact sequence
\[
\xymatrix{K_nA\ar[r] &K_nB\ar[r]& K_nC\ar[d]^{\partial}\\K_{n-1}C&K_{n-1}B\ar[l]&K_{n-1}A\ar[l]}
\]
\end{propo}
\begin{proof}
\item{i)} By \eqref{Gammatenso}, we have $\Sigma A=\Sigma\Z\otimes A$. Thus $\Sigma$ commutes with finite
products and with $M_\infty$, and sends nilpotent rings to nilpotent rings. Moreover,
$\Sigma$ is exact, because both $M_\infty$ and $\Gamma$ are. Hence the general case of i) follows
from the case $n=0$, which is proved in Section \ref{sec:basick01}.
\smallskip
\item{ii)} Consider the sequence
\[
0\to A\to \tilde{B}\to\tilde{C}\to 0
\]
Applying $\Sigma$, we obtain
\[
0\to \Sigma A\to \Sigma \tilde{B}\to\Sigma \tilde{C}\to 0
\]
By \eqref{decal}, if $D$ is any ring, then $K_0\tilde{D}=K_1\Sigma\tilde{D}$. Thus by \ref{thm:exci01} and
\ref{coro:k0split}, we get an exact sequence
\[
\xymatrix{K_0A\ar[r] &K_0B\oplus K_0\Z\ar[r]& K_0C\oplus K_0\Z\ar[d]^{\partial}\\K_{-1}C\oplus K_{-1}\Z
&K_{-1}B\oplus K_{-1}\Z\ar[l]&K_{-1}A\ar[l]}
\]
Splitting off the $K_j\Z$ summands, we obtain
\[
\xymatrix{K_0A\ar[r] &K_0B\ar[r]& K_0C\ar[d]^{\partial}\\
K_{-1}C &K_{-1}B\ar[l]&K_{-1}A\ar[l]}
\]
This proves the case $n=0$ of the proposition. The general case follows from this.\qed
\end{proof}
\begin{exa}\label{k0j}
Let $\B$ be as in Example \ref{exa:rsubi} and $I\triqui\B$ a proper ideal.
It is classical that $\cF\subset I\subset \cK$ for any such ideal. We shall show that the map
\begin{equation}\label{map:k0fk0i}
K_0\F\to K_0I
\end{equation}
is an isomorphism; thus $K_0I=K_0\F=\Z$, by Exercise \ref{exe:finrank}. As in \ref{exa:rsubi} we consider
the local ring $R_I=\C\oplus I/\F$. We have
a commutative diagram with exact rows and split exact columns
\[
\xymatrix{0\ar[r]&\cF\ar@{=}[d]\ar[r]&I\ar[d]\ar[r]&I/\cF\ar[r]\ar[d]&0\\
          0\ar[r]&\cF\ar[r]&\C\oplus I\ar[d]\ar[r]&R_I\ar[r]\ar[d]&0\\
                 &&\C\ar@{=}[r]&\C&}
\]
By \ref{exa:rsubi} and split exactness, $K_0(I/\F)=0$. Thus the map \eqref{map:k0fk0i} is onto. From the diagram above, $K_1(I/\F)\to K_0\F$
factors through $K_1(R_I)\to K_0\F$. But it follows from the discussion of Example \ref{exa:rsubi} that the map
$K_1(\C\oplus I)\to K_1(R_I)$ is onto, whence $K_1(R_I)\to K_0\F$ and thus also $K_1(I/\F)\to K_0\F$, are
zero. Thus \eqref{map:k0fk0i} is an isomorphism.
\end{exa}

\begin{exe}\label{k-1j}
A theorem of Karoubi asserts that $K_{-1}(\cK)=0$ \cite{karcomp}. Use this and excision to show that $K_{-1}(I)=0$ for any operator ideal $I$.
\end{exe}
\begin{exe}\label{exe:colineg}
Prove that if $n<0$, then the functor $K_n$ commutes with filtering colimits.
\end{exe}
\begin{rem}\label{rem:basskneg}
The definition of negative $K$-theory used here is taken from Karoubi-Villamayor's paper \cite{kv},
where cohomological notation is used.
Thus what we call $K_nA$ here is denoted $K^{-n}A$ in {\em loc. cit.} $(n\le 0)$.
There is also another definition, due to Bass \cite{bass}. A proof that the two definitions agree
is given in \cite[7.9]{kv}.
\end{rem}

\sn
\section{Topological $K$-theory}\label{sec:topk}

We saw in the previous section (Example \ref{exa:swanex}) that $K_1$ is not split exact. It follows from this that there is no way of defining higher
$K$-groups such that the long exact sequence of Theorem \ref{thm:exci01} can be extended to higher $K$-theory. This
motivates the question of whether this problem could be fixed if we replaced $K_1$ by some other functor. This is
succesfully done in topological $K$-theory of Banach algebras.
\sn
\subsection{Topological $K$-theory of Banach algebras.} A {\it Banach} ($\C$-) algebra is a $\C$-algebra together with a norm $||\ \ ||$ which makes it into a Banach space and is such
that $||xy||\le ||x||\cdot ||y||$ for all $x,y\in A$. If $A$ is a Banach algebra then its
$\C$-unitalization is the unital Banach algebra
\[
\tilde{A}_\C=A\oplus\C
\]
equipped with the product \eqref{unital_prod} and the norm $||a+\lambda||:=||a||+|\lambda|$. An algebra homomorphism is a morphism of Banach algebras
if it is continuous. If $X$ is a compact Hausdorff space and $V$ is any topological vectorspace, we write
$\cC(X,V)$ for the topological vectorspace of all continuous maps $X\to V$. If $A$ is a Banach algebra,
then $\cC(X,A)$ is again a Banach algebra with norm $||f||_\infty:=\sup_x||f(x)||$. If $X$ is locally compact, $X^+$ its one point
compactification, and $V$ a topological vectorspace, we write
\[
V(X)=\cC_0(X,V)=\{f\in\cC(X^+,V):f(\infty)=0\}.
\]
Note that if $X$ is compact, $V(X)=\cC(X,V)$. If $A$ is a Banach algebra then $A(X)$ is again a Banach algebra,
as it is the kernel of the homomorphism $\cC(X^+,A)\to A$, $f\mapsto f(\infty)$. For example, $A[0,1]$
is the algebra of continous functions $[0,1]\to A$, and $A(0,1]$ and $A(0,1)$ are identified with the ideals of $A[0,1]$
consisting of those functions which vanish
respectively at $0$ and at both endpoints. Two homomorphisms $f_0,f_1:A\to B$ of Banach algebras are called {\it homotopic} if
there exists a homomorphism $H:A\to B[0,1]$ such that the following diagram commutes.
\[
\xymatrix{&B[0,1]\ar[d]^{(\ev_0,\ev_1)}\\ A\ar[ur]^H\ar[r]_{(f_0,f_1)}&B\times B}
\]
A functor $G$ from Banach algebras to a category $\fD$ is called {\it homotopy invariant} if it maps homotopic maps to equal maps.

\begin{exe} Prove that $G$ is homotopy invariant if and only if for every Banach algebra $A$ the map $G(A)\to G(A[0,1])$ induced
by the natural inclusion $A\subset A[0,1]$ is an isomorphism.
\end{exe}

\begin{them}\label{htpyk0} (\cite[1.6.11]{rosen}) The functor
$K_0:((\mathrm{Banach\hspace{1.5pt} Algebras}))\to \ab$ is homotopy invariant.
\end{them}

\begin{exa}\label{nhtpyk1} {\it $K_1$ is not homotopy invariant}. The algebra $A:=\C[\epsilon]$ is a Banach algebra with norm $||a+b\epsilon||=|a|+|b|$. Both the inclusion $\iota:\C\to A$
and the projection $\pi:A\to \C$ are homomorphisms of Banach algebras; they satisfy $\pi\iota=1$. Moreover the map $H:A\to A[0,1]$,
$H(a+b\epsilon)(t)=a+tb\epsilon$ is also a Banach algebra homomorphism, and satisfies ${\ev}_0H=\iota\pi$, ${\ev}_1H=1$. Thus any homotopy
invariant functor $G$ sends $\iota$ and $\pi$ to inverse homomorphisms; since $K_1$ does not do so by
\ref{exa:dualn2}, it is not homotopy invariant.
\end{exa}

Next we consider a variant of $K_1$ which is homotopy invariant.

\begin{defi}\label{defi:k1top}
Let $R$ be a unital Banach algebra. Put
\[
\GL R_0:=\{g\in \GL R:\exists h\in \GL(R[0,1]): h(0)=1,h(1)=g\}.
\]
Note $\GL(R)_0\triqui\GL R$. The {\rm topological} $K_1$ of $R$ is
\[
K_1^{\top}R=\GL R /\GL(R)_0.
\]
\end{defi}
\begin{exe}
Show that if we regard $\GL R=\colim_n\GL_n R$ with the weak topology inherited from the topology of $R$, then
$K_1^{\top}R=\pi_0(\GL R)=\colim_n\pi_0(\GL_n R)$.
Then show that $K_1^{top}$ is homotopy invariant.
\end{exe}

Note that if $R$ is a unital Banach algebra, $a\in R$ and $i\ne j$, then $1+tae_{i,j}\in E(R[0,1])$ is a path connecting
$1$ to the elementary matrix $1+ae_{i,j}$. Thus $ER=[\GL(R),\GL(R)]\subset \GL(R)_0$, whence $\GL(R)_0$ is normal,
and we have a surjection
\begin{equation}\label{k1ontok1top}
K_1R\fib K_1^{\top}R.
\end{equation}
In particular, $K_1^{\top}R$ is an abelian group.

\begin{exa} Because $\C$ is a field, $K_1\C=\C^*$. Since on the other hand $\C^*$ is path connected,
we have $K_1^{\top}\C=0$.
\end{exa}

Note that $K_1^{\top}$ is additive. Thus we can extend $K_1^{\top}$ to nonunital Banach algebras in the usual way, i.e.
\[
K_1^{\top}A:=\ker(K_1^{\top}(\tilde{A}_\C)\to K_1^{\top}\C)=K_1^{\top}(\tilde{A}_\C)
\]
\sn
\begin{exe}\label{exe:k1tononuni} Show that if $A$ is a (not necessarily unital) Banach algebra, then
\[
K_1^{\top} A=\GL A/\GL(A)_0.
\]
\end{exe}

\sn

\begin{propo}\label{prop:gl0onto}(\cite[3.4.4]{black}) If $R\fib S$ is a surjective unital
homomorphism of unital Banach algebras, then $\GL (R)_0\to\GL(S)_0$ is surjective.\qed\end{propo}
\sn
Let
\begin{equation}\label{babc}
0\to A\to B\to C\to 0
\end{equation}
be an exact sequence of Banach algebras. Then
\[
0\to A\to \tilde{B}_\C\to \tilde{C}_\C\to 0
\]
is again exact. By \eqref{k1ontok1top} and \ref{prop:gl0onto}, the connecting map
$\partial:K_1(\tilde{C}_\C)\to K_0A$ of Theorem \ref{thm:exci01} sends $\ker(K_1(\tilde{C}_\C)\to K_1^{\top}C)$ to zero, and thus induces a
homomorphism
\[
\partial:K_1^{\top}C\to K_0A.
\]

\begin{them}\label{thm:excitop01}
The sequence
\[
\xymatrix{K^{\top}_1A\ar[r] &K^{\top}_1B\ar[r]& K_1^{\top}C\ar[d]^{\partial}\\K_0C&K_0B\ar[l]&K_0A\ar[l]}
\]
is exact.
\end{them}
\begin{proof} Straightforward from \ref{thm:exci01} and \ref{prop:gl0onto}.\qed
\end{proof}
Consider the exact sequences
\begin{gather*}
0\to A(0,1]\to A[0,1]\to A\to 0\\
0\to A(0,1)\to A(0,1]\to A\to 0
\end{gather*}
Note moreover that the first of these sequences is split exact.
Because $K_0$ is homotopy invariant and split exact, and because $A(0,1]$ is contractible,
we get an isomorphism
\begin{equation}\label{decato1}
K_1^{\top}A=K_0(A(0,1))
\end{equation}
Since also $K_1^{\top}$ is homotopy invariant, we put
\begin{equation}\label{decato2}
K_2^{\top} A=K_1^{\top}(A(0,1)).
\end{equation}
\begin{lem} If \eqref{babc} is exact, then
\[
0\to A(0,1)\to B(0,1)\to C(0,1)\to 0
\]
is exact too.
\end{lem}
\begin{proof}
See \cite[10.1.2]{ror} for a proof in the $C^*$-algebra case; a similar argument works for Banach algebras.\qed
\end{proof}
Taking into account the lemma above, as well as \eqref{decato1} and \eqref{decato2}, we obtain the
following corollary of Theorem \ref{thm:excitop01}.
\begin{coro}\label{excitop12}
There is an exact sequence
\[
\xymatrix{K^{\top}_2A\ar[r] &K^{\top}_2B\ar[r]& K^{\top}_2C\ar[d]^{\partial}\\K_1^{\top}C&K^{\top}_1B\ar[l]&K^{\top}_1A\ar[l]}
\]
\smartqed\qed
\end{coro}

The sequence above can be extended further by defining inductively
\[
K_{n+1}^{\top}A:=K_n^{\top}(A(0,1)).
\]
\subsection{Bott periodicity.}
Let $R$ be a unital Banach algebra. Consider the map $\beta:\idem_n R\to \GL_n \cC_0(S^1,R)$,
\begin{equation}\label{formubeta}
\beta(e)(z)=ze+1-e
\end{equation}
This map induces a group homomorphism $K_0R\to K_1^{\top}\cC_0(S^1,R)$ (see \cite[9.1]{black}). If $A$ is any Banach algebra, we
write $\beta$ for the composite
\begin{multline}\label{betato}
K_0 A\to K_0(\tilde{A}_\C)\overset\beta\to K_1^{\top}(\cC_0(S^1,\tilde{A}_\C))\\
=K_1^{\top}\C(0,1)\oplus K_1^{\top} A(0,1)\fib  K_1^{\top} A(0,1)=K_2^{\top}A
\end{multline}
One checks that for unital $A$ this defintion agrees with that given above.

\sn

\begin{them}\label{thm:bott}(Bott periodicity) (\cite[9.2.1]{black}) The composite map \eqref{betato} is an isomorphism.
\end{them}

Let \eqref{babc}
be an exact sequence of Banach algebras. By \ref{excitop12} we have a map $\partial: K_2^{\top}C\to K_1^{\top}A$.
Composing with the Bott map, we obtain a homomorphism
\[
\partial\beta:K_0 C\to K_1^{\top}A.
\]
\begin{them}\label{thm:exci_top_round}
If \eqref{babc} is an exact sequence of Banach algebras, then the sequence
\[
\xymatrix{K^{\top}_1A\ar[r] &K^{\top}_1B\ar[r]& K_1^{\top}C\ar[d]^{\partial}\\K_0C\ar[u]^{\partial\beta}&K_0B\ar[l]&K_0A\ar[l]}
\]
is exact.
\end{them}
\begin{proof} Immediate from Theorem \ref{thm:excitop01}, Corollary \ref{excitop12} and Theorem \ref{thm:bott}.\qed
\end{proof}
\subsubsection{\bf Sketch of Cuntz' proof of Bott periodicity for $C^*$-algebras.}\label{skecuntz}(\cite[Sec. 2]{cu1046})
A $C^*$-algebra is a Banach algebra $A$ equipped with additive map $*:A\to A$ such that $(a^*)^*=a$,
 $(\lambda a)^*=\bar{\lambda}a^*$, $(ab)^*=b^*a^*$ and $||aa^*||=||a||^2$ ($\lambda\in\C$, $a,b\in A$).
The {\it Toeplitz} $C^*$-algebra
is the free unital $C^*$-algebra $\cT^{\top}$ on a generator $\alpha$ subject to $\alpha\alpha^*=1$.
Since the shift $s:\ell^2(\N)\to\ell^2(\N)$, $s(e_1)=0$, $s(e_{i+1})=e_i$ satisfies $ss^*=1$, there is a homomorphism
$\cT^{\top}\to \B=\B(\ell^2(\N))$. It turns out that this is a monomorphism, that its image contains the ideal $\cK$, and that the latter
is the kernel of the $*$-homomorphism $\cT^{\top}\to \C(S^1)$ which sends $\alpha$ to the identity function $S^1\to S^1$. We have
a commutative diagram with exact rows and split exact columns:
\begin{equation}\label{diag:cuntz}
\xymatrix{0\ar[r]&\cK\ar@{=}[d]\ar[r]&\cT^{\top}_0\ar[r]\ar[d]&\C(0,1)\ar[r]\ar[d]&0\\
          0\ar[r]&\cK\ar[r]&\cT^{\top}\ar[r]\ar[d]&\C(S^1)\ar[r]\ar[d]_{\ev_1}&0\\
&&\C\ar@{=}[r]&\C&}
\end{equation}
Here we have used the identification $\cC_0(S^1,\C)=\C(0,1)$, via the exponential map; $\cT_0^{\top}$ is defined
so that the middle column be exact. Write $\overset\sim\otimes=\otimes_{\min}$ for the $C^*$-algebra
tensor product. If now $A$ is any $C^*$-algebra, and we apply
the functor $A\overset{\sim}\otimes$ to the diagram \eqref{diag:cuntz}, we obtain
a commutative diagram whose columns are split exact and whose rows are still exact (by nuclearity, see \cite[Appendix T]{wegge}).
\[
\xymatrix{0\ar[r]&A\sotimes\cK\ar@{=}[d]\ar[r]&A\sotimes\cT^{\top}_0\ar[r]\ar[d]&A(0,1)\ar[r]\ar[d]&0\\
          0\ar[r]&A\sotimes\cK\ar[r]&A\sotimes\cT^{\top}\ar[r]\ar[d]&A(S^1)\ar[r]\ar[d]&0\\
&&A\ar@{=}[r]&A&}
\]
The inclusion $\C\subset M_\infty\C\subset \cK=\cK(\ell^2(\N))$, $\lambda\mapsto \lambda e_{1,1}$ induces a natural
transformation $1\to \cK\sotimes-$; a functor $G$ from
$C^*$-algebras to abelian groups is {\it $\cK$-stable} if it is stable with respect to this data in the sense of Definition
\ref{defi:stability}. We say that $G$ is {\it half exact} if for every exact sequence \eqref{babc}, the sequence
\[
GA\to GB\to GC
\]
is exact.
\begin{rem}
In general, there is no precedence between the notions of split exact and half exact. However a functor of $C^*$-algebras
which is homotopy invariant, additive and half exact is automatically split exact (see \cite[\S21.4]{black}).
\end{rem}
\sn
\sn
The following theorem of Cuntz is stated in the literature for half exact rather than split exact functors. However the proof
uses only split exactness.
\begin{them}\label{thm:cu}(\cite[4.4]{cu1046})
Let $G$ be a functor from $C^*$-algebras to abelian groups. Assume that
\begin{itemize}
\item $G$ is homotopy invariant.
\item $G$ is $\cK$-stable.
\item $G$ is split exact.
\end{itemize}

Then for every $C^*$-algebra $A$,
\[
G(A\sotimes\cT^{\top}_0)=0.
\]
\smartqed\qed
\end{them}
\bn
\begin{propo}\label{fact:k0stable}(\cite[6.4.1]{ror}) $K_0$ is $\cK$-stable .\qed\end{propo}
\sn
It follows from the proposition above, \eqref{decato1}, Cuntz' theorem and excision, that the connecting map $\partial:K_1^{\top}(A(0,1))\to K_0(A\sotimes\cK)$
is an isomorphism. Further, one checks, using the explicit formulas for $\beta$ and $\partial$ (\eqref{formubeta},
\eqref{map:partial}), that the following diagram commutes
\[
\xymatrix{K_1^{\top}A(0,1)\ar[r]^\partial&K_0(A\sotimes\cK)\\&K_0 A\ar[ul]_\beta\ar[u]^\wr}
\]
This proves that $\beta$ is an isomorphism.
\bn

\section{Polynomial homotopy and Karoubi-Villamayor $K$-theory}\label{sec:polikv} 
In this section we analyze to what extent the results of the previous section on topological $K$-theory
of Banach algebras have algebraic analogues valid for all rings. We shall not consider continuous homotopies for general rings, among other reasons, because in general they do not carry any interesting
topologies. Instead, we shall consider polynomial homotopies. Two ring homomorphisms $f_0,f_1:A\to B$ are called
{\it elementary homotopic} if there exists a ring homomorphism $H:A\to B[t]$ such that the following diagram
commutes
\[
\xymatrix{& B[t]\ar[d]^{(\ev_0,\ev_1)}\\
          A\ar[ur]^H\ar[r]_{(f_0,f_1)}&  B}
\]
Two homomorphisms $f,g:A\to B$ are {\it homotopic} if there is a finite sequence $(f_i)_{0\le i\le n}$ of
homomophisms such that $f=f_0$, $f_n=g$, and such that for all $i$, $f_i$ is homotopic to $f_{i+1}$. We write $f\sim
g$ to indicate that $f$ is homotopic to $g$. We say that a functor $G$ from rings to a category $\fD$ is {\it homotopy
invariant} if it maps the inclusion $A\to A[t]$ ($A\in \ass$) to an isomorphism. In other words, $G$ is homotopy
invariant if it is stable (in the sense of \ref{defi:stability}) with respect to the natural inclusion $A\to A[t]$.
One checks that $G$ is homotopy invariant if and only if it preserves the homotopy relation between homomorphisms.
If $G$ is any functor, we call a ring $A$ {\it $G$-regular} if $G A\to
G(A[t_1,\dots,t_n])$ is an isomorphism for all $n$.
\begin{exam}\label{exa:k0reg} Noetherian regular rings are $K_0$-regular \cite[3.2.13]{rosen} (the same is true for
all Quillen's $K$-groups, by a result of Quillen \cite{q341}; see Schlichting's lecture notes \cite{sch}) and moreover for $n<0$, $K_n$ vanishes on such rings
(by \cite[3.3.1]{rosen} and Remark \ref{rem:basskneg}).
If $k$ is any field, then the ring $R=k[x,y]/<y^2-x^3>$ is not $K_0$-regular (this follows from
\cite[I.3.11 and II.2.3.2]{chupro}). By \ref{exa:conmut} and \ref{exa:dualn2}, the ring $k[\epsilon]$ is not $K_1$-regular;
indeed the $K_1$-class of the element $1+\epsilon t\in k[\epsilon][t]^*$ is a nontrivial element of
$\ker(K_1(k[\epsilon][t])\to K_1(k[\epsilon]))=\coker(K_1(k[\epsilon])\to K_1(k[\epsilon][t])$.
\end{exam}
The Banach algebras of paths and loops have the following algebraic analogues. Let $A$ be a ring; let
$\ev_i:A[t]\to A$ be the evaluation homomorphism ($i=0,1$). Put
\begin{gather}\label{lupez}
PA:=\ker(A[t]\overset{\ev_0}\to A)\\
\Omega A:=\ker( PA\overset{\ev_1}\to A)\label{omega}
\end{gather}
The groups $GL(\ \ )_0$ and $K_1^{\top}$ have the following algebraic analogues. Let $A$ be a unital ring. Put
\begin{align*}
\GL(A)'_0=&\ima(\GL PA\to \GL A\}\\
=&\{g\in \GL A:\exists h\in \GL(A[t]): h(0)=1,h(1)=g\}.
\end{align*}
Set
\[
KV_1A:=\GL A/\GL(A)'_0.
\]
The group $KV_1$ is the $K_1$-group of Karoubi-Villamayor \cite{kv}. It is abelian, since as we shall see
in Proposition \ref{prop:kv1ppties} below, there is a natural surjection $K_1A\fib KV_1A$. Unlike what happens with its topological
analogue, the functor $GL(\ \ )'_0$ does not preserve surjections (see Exercise \ref{exe:gl0notepi} below).
As a consequence, the $KV$-analogue of
\ref{thm:exci01} does not hold for general short exact sequences of rings, but only for those sequences \eqref{abc}
such that $\GL(B)'_0\to \GL(C)'_0$ is onto, such as, for example, split exact sequences.
Next we list some of the basic properties of $KV_1$; all except nilinvariance (due independently to
Weibel \cite{wn} and Pirashvili \cite{pira}) were proved by Karoubi and Villamayor in \cite{kv}.
\begin{prop}\label{prop:kv1ppties}
\item{i)} There is a natural surjective map $K_1 A\fib KV_1A$ $(A\in\ass)$.
\item{ii)} The rule $A\mapsto KV_1A$ defines a split-exact functor $\ass\to\ab$.
\item{iii)} If \eqref{abc} is an exact sequence such that the map $\GL(B)'_0\to \GL(C)'_0$ is onto, then
the map $K_1C\to K_0A$ of Theorem \ref{thm:exci01} factors through $KV_1 C$, and the resulting sequence
\[
\xymatrix{KV_1A\ar[r] &KV_1B\ar[r]& KV_1C\ar[d]^{\partial}\\K_0C&K_0B\ar[l]&K_0A\ar[l]}
\]
is exact.
\item{iv)}$KV_1$ is additive, homotopy invariant, nilinvariant and $M_\infty$-stable.
\end{prop}
\begin{proof}
If \eqref{abc} is exact and $\GL(B)'_0\to \GL(C)'_0$ is onto, then it is clear that
\begin{equation}\label{seq:kv1}
KV_1A\to KV_1 B\to KV_1C
\end{equation}
is exact, and we have a commutative diagram with exact rows and columns
\begin{equation}\label{diag:kv}
\xymatrix{&1\ar[d]&1\ar[d]&1\ar[d]\\
1\ar[r]&\GL\Omega A\ar[d]\ar[r]&\GL\Omega B\ar[d]\ar[r]&\GL\Omega C\ar[d]\\
1\ar[r]&\GL PA\ar[d]\ar[r]&\GL PB\ar[d]\ar[r]&\GL PC\ar[d]\\
1\ar[r]&\GL A\ar[r]&\GL B\ar [r]&\GL C}
\end{equation}
If moreover \eqref{abc} is split exact, then each row in the diagram above is, and one checks, by looking at this
diagram, that $\GL(A)'_0=\GL(B)'_0\cap \GL(A)$, whence $KV_1 A\to KV_1B$ is injective. Thus
\begin{equation}\label{seq:kv1rightsplit}
0\to KV_1A\to KV_1B\to KV_1C\to 0
\end{equation}
is exact. In particular
\begin{equation}\label{kv1nunit}
KV_1 A=\ker(KV_1\tilde{A}\to KV_1\Z).
\end{equation}
If $R$ is unital, then $\GL(R)'_0\supset E(R)$, by the same argument as in the Banach algebra case \eqref{k1ontok1top}.
In particular $KV_1 R$ is abelian. This proves the unital case of i);
the general case follows from the unital one using \eqref{kv1nunit}. Thus the right split exact sequence \eqref{seq:kv1rightsplit}
is in fact a split exact sequence of abelian groups. It follows that $KV_1$ is split exact, proving ii).
Part iii) follows from part i) and \eqref{seq:kv1}.
The proof that $KV_1$ is $M_\infty$-stable on unital rings is the same as the proof that $K_1$ is. By split exactness, it follows that $KV_1$ is $M_\infty$-stable
on all rings. To prove that $KV_1$ is homotopy invariant, we must show that the split surjection
$\ev_0:KV_1(A[t])\to KV_1 A$ is injective. By split exactness, its kernel is $KV_1PA=\GL PA/\GL(PA)'_0$, so we must prove that
$\GL PA\subset \GL(PA)'_0$. But if $\alpha(s)\in \GL PA$, then $\beta(s,t):=\alpha(st)\in \GL PPA$ and
$\ev_{t=1}(\beta)=\alpha$. Thus homotopy invariance is proved. If \eqref{abc} is exact and $A$ is nilpotent, then
$PA$ and $\Omega A$ are nilpotent too, whence all those maps displayed in diagram \eqref{diag:kv} which are induced by $B\to C$
are surjective. Diagram chasing shows that $\GL(B)_0'\to \GL(C)'_0$ is surjective, whence by iii) we have an exact sequence
\[
KV_1A\to KV_1B\to KV_1C\to 0
\]
Thus to prove $KV_1$ is nilinvariant, it suffices to show that if $A^2=0$ then $KV_1A=0$. But if $A^2=0$,
then the map $H:A\to A[t]$, $H(a)=at$ is a ring homomorphism, and satisfies $\ev_0H=0$, $\ev_1H=1_A$. Hence
$KV_1A=0$, by homotopy invariance.\qed
\end{proof}
Consider the exact sequence
\begin{equation}\label{lopez}
0\to \Omega A\to PA\overset{\ev_1}\to A\to 0
\end{equation}
By definition, $\GL(A)'_0=\ima(\GL(PA)\to \GL(A))$. But in the course of the proof of Proposition \ref{prop:kv1ppties}
above, we have shown that $\GL(PA)=\GL(PA)'_0$, so by \ref{prop:kv1ppties} iii), we have a natural  map
\begin{equation}\label{kv1k0}
KV_1(A)\overset\partial\to K_0(\Omega A).
\end{equation}
Moreover, \eqref{kv1k0} is injective, by \ref{prop:kv1ppties} iii) and iv).
This map will be of use in the next section.\goodbreak
Higher $KV$-groups are defined by iterating the loop functor $\Omega$:
\[
KV_{n+1}(A)=KV_1(\Omega^nA).
\]
Higher $KV$-theory satisfies excision for those exact sequences $\eqref{abc}$ such that
 for every $n$, the map
\[
\{\alpha\in \GL(B[t_1,\dots,t_n]):\alpha(0,\dots,0)=1\}\to \{\alpha\in \GL(C[t_1,\dots,t_n]):\alpha(0,\dots,0)=1\}
\]
is onto. Such sequences are usually called $\GL$-fibration sequences (a different notation was used in \cite{kv}). Note that
if \eqref{abc} is a $\GL$-fibration, then $\GL(B)'_0\to\GL(C)'_0$ is surjective, and thus \ref{prop:kv1ppties} iii)
applies. Moreover it is proved in \cite{kv} that if \eqref{abc} is a $\GL$-fibration,
then there is a long exact sequence $(n\ge 1)$
\[
\xymatrix{KV_{n+1}B\ar[r]& KV_{n+1}C\ar[r]&KV_n(A)\ar[r]&KV_n(B)\ar[r]&KV_n(C).}
\]

\section{Homotopy $K$-theory}\label{sec:kh}
\subsection{Definition and basic properties of $KH$.}
Let $A$ be a ring. Consider the natural map
\begin{equation}\label{map:delukh}
\partial:K_0A\to K_{-1}\Omega A
\end{equation}
associated with the exact sequence \eqref{omega}. Since $K_{-p}=K_0\Sigma^p$, we may iterate the construction and form the colimit
\begin{equation}\label{kh0}
KH_0A:=\colim_p K_{-p}\Omega^pA.
\end{equation}
Put
\begin{equation}\label{khn}
KH_nA:=\colim_pK_{-p}\Omega^{n+p}A=\left\{\begin{array}{cc}KH_0\Omega^nA& (n\ge 0)\\ KH_0\Sigma^nA&(n\le 0)\end{array}\right.
\end{equation}
The groups $KH_*A$ are Weibel's {\it homotopy $K$-theory} groups of $A$ (\cite{wh},\cite[8.1.1]{biva}). One
can also express $KH$ in terms of $KV$, as we shall see presently. We need some preliminaries first.
We know that $K_1(S)=0$ for every infinite sum ring $S$; in particular $KV_1(\Gamma R)=0$
for unital $R$, by \ref{prop:kv1ppties} i). Using split exactness of $KV_1$, it follows that $KV_1\Gamma A=0$ for every ring $A$. Thus the dotted
arrow in the commutative diagram with exact row below exists
\[
\xymatrix{K_1(\Gamma A)\ar[r]&K_1(\Sigma A)\ar@{>>}[d]\ar[r]&K_0(A)\ar@{.>}[dl]\ar[r]& 0\\
                    & KV_1(\Sigma A)&&}
\]
The map $K_1(\Sigma A)\to KV_1(\Sigma A)$ is surjective by Proposition \ref{prop:kv1ppties} i). Thus the dotted
arrow above is a surjective map
\begin{equation}\label{map:k0kv1sigma}
K_0(A)\fib KV_1(\Sigma A).
\end{equation}
On the other hand, the map \eqref{kv1k0} applied to $\Sigma A$ gives
\begin{equation}\label{map:kv1sigmak-1omega}
KV_1(\Sigma A)\overset\partial\to K_0(\Omega\Sigma A)=K_0(\Sigma \Omega A)=K_{-1}(\Omega A)
\end{equation}
One checks, by tracking down boundary maps, (see the proof of \cite[8.1.1]{biva}) that the composite of \eqref{map:k0kv1sigma} with \eqref{map:kv1sigmak-1omega}
is the map \eqref{map:delukh}:
\begin{equation}\label{diag:khkh}
\xymatrix{K_0(A)\ar[rr]^{\eqref{map:delukh}}\ar[dr]_{\eqref{map:k0kv1sigma}}&& K_{-1}(\Omega A)\\
                        & KV_1(\Sigma A)\ar[ur]_{\eqref{map:kv1sigmak-1omega}}&}
\end{equation}
On the other hand,  \eqref{map:kv1sigmak-1omega} followed by \eqref{map:k0kv1sigma} applied
to $\Sigma\Omega A$ yields a map
\[
KV_1(\Sigma A)\to KV_1(\Sigma^2\Omega A)=KV_2(\Sigma^2A).
\]
Iterating this map one obtains an inductive system; by \eqref{diag:khkh}, we get
\[
KH_0(A)=\colim_rKV_1(\Sigma^{r+1}\Omega^r A)=\colim_rKV_r(\Sigma^{r}A)
\]
and in general,
\begin{equation}\label{khkv}
KH_n(A)=\colim_r KV_1(\Sigma^{r+1}\Omega^{n+r}A)=\colim_rKV_{n+r}(\Sigma^{r}A).
\end{equation}
Next we list some of the basic properties of $KH$, proved by Weibel in \cite{wh}.
\begin{them}\label{thm:khppties} (\cite{wh})
Homotopy $K$-theory has the following properties.
\begin{itemize}
\item[i)] It is homotopy invariant, nilinvariant and $M_\infty$-stable.
\item[ii)] It satisfies excision: to the sequence \eqref{abc} there corresponds a long exact sequence ($n\in\Z$)
\[
KH_{n+1}C\to KH_nA\to KH_nB\to KH_nC\to KH_{n-1}A.
\]
\end{itemize}
\end{them}
\begin{proof}
From \eqref{khkv} and the fact that, by \ref{prop:kv1ppties}, $KV$ is homotopy invariant, it follows that $KH$ is homotopy invariant.
Nilinvariance, $M_\infty$-stability and excision for $KH$ follow from
the fact that (by \ref{prop:kneg}) these hold for nonpositive $K$-theory, using the formulas \eqref{kh0}, \eqref{khn}.\qed
\end{proof}

\begin{exe}\label{exe:htpinv}
Note that in the proof of \ref{thm:khppties}, the formula \eqref{khkv} is used only for homotopy invariance.
Prove that $KH$ is homotopy invariant without using \eqref{khkv}, but using excision instead. Hint: show
that the excision map $KH_*(A)\to KH_{*-1}(\Omega A)$ coming from the sequence \eqref{lopez} is an isomorphism.
\end{exe}
\begin{exe} Show that $KH$ commutes with filtering colimits.
\end{exe}
\subsection{$KH$ for $K_0$-regular rings.}

\begin{lem}\label{lem:pareg}
Let $A$ be a ring. Assume that $A$ is $K_n$-regular for all $n\le 0$. Then $KV_1(A)\to K_0(\Omega A)$ is an isomorphism, and for $n\le 0$,
$K_n(PA)=0$, $PA$ and $\Omega A$ are $K_n$-regular, and $K_nA\to K_{n-1}\Omega A$ an isomorphism.
\end{lem}
\begin{proof}
Consider the split exact sequence
\[
\xymatrix{0\ar[r]&PA[t_1,\dots,t_r]\ar[r]&A[s,t_1,\dots,t_r]\ar[r]&A[t_1,\dots,t_r]\ar[r]&0}
\]
Applying $K_n$ ($n\le 0$) and using that $K_n$ is split exact and that, by hypothesis, $A$ is $K_n$-regular, we get that $K_n(PA[t_1,\dots,t_r])=0$.
As this happens for all $r\ge 0$, $PA$ is $K_n$-regular. Hence the map
of exact sequences
\[
\xymatrix{0\ar[r]&\Omega A\ar[d]\ar[r]&PA\ar[d]\ar[r]&A\ar[d]\ar[r]&0\\
          0\ar[r]&\Omega A[t_1,\dots,t_r]\ar[r]&PA[t_1,\dots,t_r]\ar[r]&A[t_1,\dots,t_r]\ar[r]&0}
\]
induces commutative squares with exact rows
\[
\xymatrix{0\ar[r]&KV_1(A)\ar[d]\ar[r]&K_0(\Omega A)\ar[d]\ar[r]&0\\
          0\ar[r]&KV_1(A[t_1,\dots,t_r])\ar[r]&K_0(\Omega A[t_1,\dots,t_r])\ar[r]&0}
\]
and
\[
\xymatrix{0\ar[r]&K_n(A)\ar[d]\ar[r]&K_{n-1}(\Omega A)\ar[d]\ar[r]&0\\
          0\ar[r]&K_n(A[t_1,\dots,t_r])\ar[r]&K_{n-1}(\Omega A[t_1,\dots,t_r])\ar[r]&0}\qquad (n\le 0)
\]
By Proposition \ref{prop:kv1ppties} and our hypothesis, the first vertical map in each diagram is an isomorphism; it follows that the second is also an isomorphism.\qed
\end{proof}
\begin{rem}\label{rem:vorst} A theorem of Vorst \cite{vorst} implies that if $A$ is $K_0$-regular then it is
$K_n$-regular for all $n\le 0$. Thus the lemma above holds whenever $A$ is $K_0$-regular. The statement of
Vorst's theorem is that, for Quillen's $K$-theory, and $n\in\Z$, a $K_n$-regular unital ring is also
$K_{n-1}$-regular. (In his paper, Vorst states this only for commutative rings, but his proof works in general).
For $n\le 0$, Vorst's theorem extends to all, not necessarily unital rings. To see this, one shows first, using the fact that
$\Z$ is $K_n$-regular (since it is noetherian regular), and split exactness, that $A$ is $K_n$-regular if and only if $\tilde{A}$ is.
Now Vorst's theorem applied to $\tilde{A}$ implies that if $A$ is $K_n$-regular then it is $K_{n-1}$-regular
$(n\le 0)$.
\end{rem}

\begin{propo}\label{prop:khk0reg}
If $A$ satisfies the hypothesis of Lemma \ref{lem:pareg}, then
\[
KH_n(A)=\begin{cases} KV_n(A) & n\ge 1\\
                      K_n(A) & n\le 0\end{cases}
\]
\end{propo}
\begin{proof} By the lemma, $KV_{n+1}(A)=KV_1(\Omega^nA)\to K_0(\Omega^{n+1}A)$ and $K_{-n}(\Omega^pA)\to K_{-n-1}(\Omega^{p+1}A)$ are isomorphisms
for all $n,p\ge 0$.\qed
\end{proof}

\subsection{Toeplitz ring and the fundamental theorem for $KH$.}
Write $\cT$ for the free unital ring on two generators $\alpha$, $\alpha^*$ subject to $\alpha\alpha^*=1$. Mapping $\alpha$
to $\sum_ie_{i,i+1}$ and $\alpha^*$ to $\sum_ie_{i+1,i}$ yields a homomorphism $\cT\to\Gamma:=\Gamma\Z$ which is injective
(\cite[Proof of 4.10.1]{biva}); we identify $\cT$ with its image in $\Gamma$. Note
\begin{equation}\label{eijintoep}
{\alpha^*}^{p-1}\alpha^{q-1}-{\alpha^*}^{p}\alpha^{q} = e_{p,q}\qquad (p,q\ge 1).
\end{equation}
 Thus $\cT$ contains
the ideal $M_\infty:=M_\infty\Z$. There is a commutative diagram with exact rows and split exact columns:
\[
\xymatrix{0\ar[r]&M_\infty\ar@{=}[d]\ar[r]&\cT_0\ar[d]\ar[r]& \sigma\ar[d]\ar[r]&0\\
          0\ar[r]&M_\infty\ar[r]&\cT\ar[d]\ar[r]& \Z[t,t^{-1}]\ar[d]^{\ev_1}\ar[r]&0\\
   &&\Z\ar@{=}[r]&\Z&}
\]
Here the rings $\cT_0$ and $\sigma$ of the top row are defined so that the columns be exact. Note moreover that the rows are split as sequences of abelian groups. Thus tensoring with any ring $A$ yields an exact diagram
\begin{equation}\label{toepalg}
\xymatrix{0\ar[r]&M_\infty A\ar@{=}[d]\ar[r]&\cT_0A\ar[d]\ar[r]& \sigma A\ar[d]\ar[r]&0\\
          0\ar[r]&M_\infty A\ar[r]&\cT A\ar[d]\ar[r]& A[t,t^{-1}]\ar[d]\ar[r]&0\\
   & &A\ar@{=}[r]&A&}
\end{equation}
Here we have omitted tensor products from our notation; thus $\cT A=\cT\otimes A$, $\sigma A=\sigma\otimes A$,
and $\cT_0A=\cT_0\otimes A$.
We have the following algebraic analogue of Cuntz' theorem.
\begin{them}\label{thm:algcuntz} (\cite[7.3.2]{biva})
Let $G$ be a functor from rings to abelian groups. Assume that:
\begin{itemize}
\item $G$ is homotopy invariant.
\item $G$ is split exact.
\item $G$ is $M_\infty$-stable.
\sn Then for any ring $A$, $G(\cT_0A)=0$.\qed
\end{itemize}
\end{them}

\begin{them}\label{thm:decalkh}
Let $A$ be a ring and $n\in\Z$. Then
\[
KH_n(\sigma A)=KH_{n-1}(A).
\]
\end{them}
\begin{proof}
By Theorem \ref{thm:khppties}, $KH$ satisfies excision. Apply this to the top row
of diagram \eqref{toepalg} and use Theorem \ref{thm:algcuntz}.\qed
\end{proof}

The following result of Weibel \cite[1.2 iii)]{wh} is immediate from the theorem above.

\begin{coro}\label{coro:khfund}(Fundamental theorem for $KH$, \cite[1.2 iii)]{wh}).
$KH_n(A[t,t^{-1}])=KH_n(A)\oplus KH_{n-1}(A)$.\qed
\end{coro}

\begin{rem}\label{rem:khfund}
We regard Theorem \ref{thm:decalkh} as an
algebraic analogue of Bott periodicity for $KH$. What is missing in the algebraic case is an analogue
of the exponential map; there is no isomorphism
$\Omega A\to \sigma A$.
\end{rem}

\begin{rem}\label{rem:kvnotfund}
We have shown in Proposition \ref{prop:kv1ppties} that $KV_1$ satisfies the hypothesis of Theorem
\ref{thm:algcuntz}. Thus $KV_1(\cT_0A)=0$ for every ring $A$. However, there is no natural isomorphism $KV_1(\sigma A)=K_0(A)$. Indeed,
since $KV_1(\sigma PA)=KV_1(P\sigma A)=0$ the existence of such an isomorphism would imply that $K_0(PA)=0$, which in turn,
given the fact that $K_0$ is split exact, would imply that $K_0(A[t])=K_0(A)$, a formula which does not hold for general
$A$ (see Example \ref{exa:k0reg}).
\end{rem}
We finish the section with a technical result which will be used in Subsection \ref{subsec:fundatoep}.
\begin{propo}\label{prop:atotau=0}
Let $\fD$ be an additive category, $F:\ass\to \fD$ an additive functor, and $A$ a ring. Assume $F$ is $M_\infty$-stable on $A$ and $M_2$-stable
on both $\cT A$
and $M_2(\cT A)$. Then $F(M_\infty A\to \cT A)$ is the zero map.
\end{propo}
\begin{proof}
Because $F$ is $M_\infty$-stable on $A$, it suffices to show that $F$ sends the inclusion $\j:A\to \cT A$, $a\mapsto ae_{11}$,
to the zero map. Note that $e_{11}=1-\alpha^*\alpha$, by \eqref{eijintoep}. Consider the inclusion $\j^{\infty}:A\to \cT A$,
$\j^{\infty}(a)=a\cdot 1=\diag(a,a,a,\dots)$. One checks that the following matrix
\[
Q=\left[\begin{matrix}1-\alpha^*\alpha & &\alpha^*\\\alpha & & 0\end{matrix}\right]\in \GL_2\cT(\tilde{A})
\]
satisfies $Q^2=1$ and
\[
Q\left[\begin{array}{cc}\j(a)&0\\0&\j^\infty(a)\end{array}\right]Q=\left[\begin{array}{cc}\j^\infty(a)&0\\0&0\end{array}\right].
\]
Since we are assuming that $F$ is additive and $M_2$-stable on both $\cT A$ and $\M_2\cT A$, we may now apply Exercise
\ref{exe:add} and Proposition \ref{prop:vw} to deduce the following identity between elements of the group $\hom_\fD(F(A),F(\cT A))$:
\[
\j+\j^{\infty}=\j^{\infty}.
\]
It follows that $\j=0$, as we had to prove.\qed
\end{proof}

\begin{exe}\label{exe:gl0notepi}
Deduce from Remark \ref{rem:kvnotfund} and propositions \ref{prop:atotau=0} and \ref{prop:kv1ppties} iii)
that the canonical maps $\GL(\cT_0 A)'_0\to \GL(\sigma A)'_0$ and $\GL(\cT A)'_0\to \GL(A[t,t^{-1}])'_0$
are not surjective in general.
\end{exe}

\section{Quillen's Higher $K$-theory}\label{sec:kq}

Quillen's higher $K$-groups of a unital ring $R$ are defined as the homotopy groups of a certain $CW$-complex;
the plus construction of the classifying space of the group $\GL(R)$ \cite{qplus}. The latter construction is defined more generally
for $CW$-complexes, but we shall not go into this general version; for this and other matters connected with the plus
construction approach to $K$-theory, the interested reader should consult standard references such as Jon Berrick's book
\cite{berr}, or the papers of Loday \cite{lorep}, and Wagoner \cite{rosen}. We shall need
a number of basic facts from algebraic topology, which we shall presently review. First of all we recall that if $X$
and $Y$ are $CW$-complexes, then the cartesian product $X\times Y$, equipped with the product topology, is not a $CW$-complex
in general. That is, the category of $CW$-complexes is not closed under finite products in $\topo$. On the other hand,
any $CW$-complex is a compactly generated --or Kelley-- space, and the categorical product of two $CW$-complexes
in the category $Ke$ of compactly generated spaces is again $CW$, and also has the cartesian product $X\times Y$ as
underlying set. Moreover, in case the product topology in $X\times Y$ happens to be compactly generated, then it
agrees with that of the product in $Ke$. In these notes, we write $X\times Y$ for the cartesian product equipped
with its compactly generated topology. (For a more detailed treatment of the categorical properties of $Ke$ see \cite{gs}).

\mn

\subsection{Classifying spaces.}
The {\it classifying space} of a (discrete) group $G$ is a pointed connected $CW$-complex $BG$ such that
\[
\pi_nBG=\left\{\begin{array}{cc}G& n=1\\ 0 &n\ne 1\end{array}\right.
\]
This property characterizes $BG$ and makes it functorial up to homotopy. Further there are various strictly functorial models
for $BG$ (\cite[Ch. 5\S1]{rosen}, \cite[1.5]{goja}). We choose the model coming from the realization of the simplicial nerve of $G$
(\cite{goja}), and write $BG$ for that model. Here are some basic
properties of $BG$ which we shall use.
\begin{properties}\label{fact:bg}
\item{i)} If
\[
1\to G_1\to G_2\to G_3\to 1
\]
is an exact sequence of groups, then $BG_2\to BG_3$ is a fibration with fiber $BG_1$.
\item{ii)} If $G_1$ and $G_2$ are groups, then the map $B(G_1\times G_2)\to BG_1\times BG_2$ is a homeomorphism.

\item{iii)} The homology of $BG$ is the same as the group homology of $G$; if $M$ is
$\pi_1BG=G$-module, then
\[
H_n(BG,M)=H_n(G,M):=\tor^{\Z G}_n(\Z,M)
\]
\smartqed\qed
\end{properties}
\subsection{Perfect groups and the plus construction for $BG$.}
A group $P$ is called {\it perfect} if its abelianization is trivial, or equivalently, if $P=[P,P]$. Note that a group
$P$ is perfect if and only if the functor $\hom_{\Grp}(P,-):\ab\to \ab$ is zero. Thus the full subcategory $\subset \Grp$ of
all perfect groups is closed both under colimits and under homomorphic images. In particular, if $G$ is group, then the directed set of all
perfect subgroups of $G$ is filtering, and its union is again a perfect subgroup $N$, the maximal perfect subgroup of $G$.
Since the conjugate of a perfect subgroup is again perfect, it follows that $N$ is normal in $G$. Note that $N\subset [G,G]$;
if moreover the equality holds, then we say that $G$ is {\it quasi-perfect}. For example, if $R$ is a unital ring, then
$\GL R$ is quasi-perfect, and $ER$ is its maximal perfect subgroup (\cite[2.1.4]{rosen}).
Quillen's {\it plus construction} applied to the group $G$ yields a cellular map of $CW$-complexes
$\iota:BG\to (BG)^+$ with the following properties (see \cite[1.1.1, 1.1.2]{lorep}).
\begin{itemize}
\item[i)] At the level of $\pi_1$, $\iota$ induces the projection $G\fib G/N$. 
\item[ii)] At the level of homology, $\iota$ induces an isomorphism $H_*(G,M)\to H_*((BG)^+,M)$ for each $G/N$-module $M$.
\item[iii)] If $BG\to X$ is any continuous function which at the level of $\pi_1$ maps $N\to 1$, then the dotted arrow in the following
diagram exists and is unique up to homotopy
\[
\xymatrix{BG\ar[r]^\iota\ar[d]&(BG)^+\ar@{.>}[dl]\\X&}
\]
\item[iv)] Properties i) and iii) above characterize $\iota:BG\to BG^+$ up to homotopy.
\end{itemize}
From the universal property, it follows that if $f:BG_1\to BG_2$ is a continuous map, then there is a (continuous) map
$BG_1^+\to BG_2^+$, unique up to homotopy, which makes the following diagram commute
\[
\xymatrix{BG_1\ar[d]\ar[r]^f&BG_2\ar[d]\\ BG_1^+\ar[r]^{f^+}&BG_2^+}
\]

\begin{properties}\label{facts:bg+}(\cite[1.1.4]{lorep})
\sn
\item{i)} If $G_1$ and $G_2$ are groups, and $\pi_i:B(G_1\times G_2)\to BG_i$ is the projection, then the map
$(\pi_1^+,\pi_2^+):B(G_1\times G_2)^+\to BG_1^+\times BG_2^+$ is a homotopy equivalence.
\item{ii)} The map $BN^+\to BG^+$ is the universal classifying space of $BG^+$.\qed
\end{properties}

If
\begin{equation}\label{seq:g123}
1\to G_1\to G_2\overset{\pi}\to G_3\to 1
\end{equation}
is an exact sequence of groups, then we can always choose $\pi^+$ to be
a fibration; write $F$ for its fiber. If the induced map $G_1\to \pi_1F$ kills the maximal perfect subgroup $N_1$ of $G_1$,
then $BG_1\to F$ factors through a map
\begin{equation}\label{map:fibplus}
BG_1^+\to F
\end{equation}

\begin{propo}\label{prop:fibplus}
Let \eqref{seq:g123} be an exact sequence of groups. Assume that
\sn
\item{i)} $G_1$ is quasi-perfect and $G_2$ is perfect.
\item{ii)}$G_3$ acts trivially on $H_*(G_1,\Z)$.
\item{iii)}$\pi_1F$ acts trivially on $H_*(F,\Z)$.\goodbreak
\sn
Then the map \eqref{map:fibplus} is a homotopy equivalence.
\end{propo}
\begin{proof}
Consider the map of fibration sequences
\[
\xymatrix{BG_1\ar[d]\ar[r]&BG_2\ar[d]\ar[r]&BG_3\ar[d]\\
           F\ar[r]&BG_2^+\ar[r]& BG_3^+}
\]
By the second property of the plus construction listed above, the maps $BG_i\to BG_i^+$ are homology equivalences.
 For $i\ge 2$, we have, in addition, that $G_i$ is perfect, so $BG_i^+$ is simply connected and $F$ is connected with
abelian $\pi_1$, isomorphic to $\coker(\pi_2BG_2^+\to\pi_2BG_3^+)$. Hence $\pi_1F\to H_1F$ is an isomorphism, by
Poincar\'e's theorem. All this together with the Comparison Theorem (\cite{zee}),
imply that $BG_1\to F$ and thus also \eqref{map:fibplus}, are homology equivalences. Moreover, because $G_1$ is
quasi-perfect by hypothesis, the Hurewicz map $\pi_1BG_1^+\to H_1BG_1^+$ is an isomorphism, again by Poincar\'e's
theorem. Summing up, $BG_1^+\to F$ is a homology isomorphism which induces an isomorphism of fundamental groups;
since $\pi_1F$ acts trivially on $H_*F$ by hypothesis, this implies that \eqref{map:fibplus} is a weak equivalence
(\cite[4.6.2]{brow}).\qed
\end{proof}

\begin{lem}\label{lem:fibplus}
Let \eqref{seq:g123} be an exact sequence of groups. Assume that for every $g\in G_2$ and every finite set $h_1,\dots,h_k$ of elements
of $G_1$, there exists an $h\in G_1$ such that for all $i$, $gh_ig^{-1}=hh_ih^{-1}$. Then $G_3$ acts trivially on $H_*(G_1,\Z)$.
\end{lem}
\begin{proof}
If $g\in G_2$ maps to $\bar{g}\in G_3$, then the action of $\bar{g}$ on $H_*(G_1,\Z)$ is that induced by conjugation by $g$.
The hypothesis implies that the action of $g$ on any fixed cycle of the standard bar complex which computes $H_*(BG_1,\Z)$
(\cite[6.5.4]{chubu}) coincides with the conjugation action by an element of $G_1$, whence it is trivial (\cite[6.7.8]{chubu}).\qed
\end{proof}

Quillen's higher $K$-groups of a unital ring $R$ are defined as the homotopy groups of $(B\GL R)^+$; we put
\begin{align*}
K(R):&=(B\GL R)^+\\
K_nR:&=\pi_nK (R)\qquad (n\ge 1).
\end{align*}
In general, for a not necessarily unital ring $A$, we put
\[
K(A):=\mathrm{fiber}(K(\tilde{A})\to K(\Z)), \qquad K_n(A)=\pi_nK(A)\qquad (n\ge 1)
\]
One checks, using \ref{facts:bg+} i), that when $A$ is unital, these definitions agree with the previous ones.
\begin{rem} As defined, $K$ is a functor from $\ass$ to the homotopy category of topological spaces $\joto$.
Further note that for $n=1$ we recover the definition of $K_1$ given in \ref{defi:k01}. \end{rem}

We shall see below that the main basic properties of Section \ref{sec:basick01} which hold for $K_1$ hold also for
higher $K_n$ (\cite{lorep}). First we need some preliminaries. If $W:\N\to \N$ is an injection, we shall identify
$W$ with the endomorphism $\Z^{(\N)}\to \Z^{(\N)}$, $W(e_i)=e_{w(i)}$ and also with the matrix of the latter in the
canonical basis, given by  $W_{ij}=\delta_{i,w(j)}$. Let $V=W^t$ be the transpose matrix; then $VW=1$. If now $R$ is
a unital ring, then the endomorphism $\psi^{V,W}:M_\infty R\to M_\infty R$ of \ref{prop:vw} induces a group
endomorphism $\GL(R)\to\GL(R)$, which in turn yields homotopy classes of
maps
\begin{equation}\label{map:psi+}
\psi:K(R)\to K(R),\qquad \psi':BE(R)^+\to BE(R)^+.
\end{equation}

\begin{lem}\label{lem:psitriv} (\cite[1.2.9]{lorep}) The maps \eqref{map:psi+} are homotopic to the identity.\qed
\end{lem}

A proof of the previous lemma for the case of $\psi$ can be found in {\it loc. cit.}; a similar argument
works for $\psi'$.

\begin{propo}\label{prop:hikay}
Let $n\ge 1$ and let $R$ be a unital ring.
\sn
\item{i)} The functor $K_n:\ass_1\to \ab$ is additive.
\item{ii)} The direct sum $\oplus:\GL R\times \GL R\to \GL R$ of \eqref{directsum} induces a map $K(R)\times K(R)\to K(R)$
which makes $K(R)$ into an $H$-group, that is, into a group up to homotopy. Similarly, $BE(R)^+$ also has an $H$-group
structure induced by $\oplus$.

\item{iii)} The functors $K:\ass_1\to\joto$ and $K_n:\ass_1\to\ab$ are $M_\infty$-stable.
\end{propo}
\begin{proof}
Part i) is immediate from \ref{facts:bg+} i). The map of ii) is the composite of the homotopy inverse of the map of
\ref{facts:bg+} i) and the map $B\GL(\oplus)^+$. One checks that, up to endomorphisms of the form $\psi^{V,W}$
induced by injections $\N\to \N$, the map $\oplus:\GL(R)\times\GL(R)$ is associative and commutative and the
identity matrix is a neutral element for $\oplus$. Hence by \ref{lem:psitriv} it follows that $B\GL(R)^+$ is a
commutative and associative $H$-space. Since it is connected, this implies that it is an $H$-group, by
\cite[X.4.17]{white}. The same argument shows that $BE(R)^+$ is also an $H$-group. Thus ii) is proved. Let
$\iota:R\to M_\infty R$ be the canonical inclusion. To prove iii), one observes that a choice of bijection $\N\times
\N\to \N$ gives an isomorphism $\phi:M_\infty M_\infty R\iso M_\infty R$ such that the composite with
$M_\infty\iota$ is a homomorphism the form $\psi^{V,W}$ for some injection $W:\N\to\N$, whence the induced map
$K(R)\to K(R)$ is homotopic to the identity, by Lemma \ref{lem:psitriv}. This proves that $K(\iota)$ is a homotopy
equivalence.\qed
\end{proof}
\begin{coro}\label{coro:ksumtriv}
If $S$ is an infinite sum ring, then $K(S)$ is contractible.
\end{coro}
\begin{proof} It follows from the theorem above, using Exercise \ref{exe:matrix} ii) and Proposition \ref{prop:sumring}.\qed
\end{proof}

\begin{propo}\label{prop:conseq}
Let $R$ be a unital ring, $\Sigma R$ the suspension, $\Omega K(\Sigma R)$ the loopspace, and
$\Omega_0K(\Sigma R)\subset \Omega K(\Sigma R)$ the connected component
of the trivial loop. There is a homotopy equivalence
\[
K(R)\weq \Omega_0K(\Sigma R).
\]
\end{propo}
\begin{proof}
Consider the exact sequence of rings
\[
0\to M_\infty R\to \Gamma R\to \Sigma R\to 0
\]
Since $K_1(\Gamma R)=0$, we have
\[
\GL(\Gamma R)=E(\Gamma R),
\]
which applies onto $E\Sigma R$. Thus we have an exact sequence of groups
\[
1\to \GL M_\infty R\to \GL\Gamma R\to E\Sigma R\to 1
\]
One checks that the inclusion $\GL(M_\infty R)\to \GL(\Gamma R)$
satisfies the hypothesis of \ref{lem:fibplus} (see \cite[bottom of page 357]{wa} for details).
Thus the perfect group $E(\Sigma R)$ acts trivially on $H_*(\GL M_\infty R,\Z)$. On the other
hand, by \ref{prop:hikay} ii), both $K(\Gamma R)$ and $BE(\Sigma R)^+$ are $H$-groups, and moreover
since $\pi:\GL(\Gamma R)\to \GL(\Sigma R)$ is compatible with $\oplus$, the map $\pi^+:K(\Gamma R)\to BE(\Sigma R)^+$
can be chosen to be compatible with the induced operation. This implies that the fiber  of $\pi^+$ is a connected $H$-space
(whence an $H$-group) and so its fundamental group acts trivially on its homology. 
Hence by Propositions \ref{prop:hikay} iii) and \ref{prop:fibplus}, we have a homotopy fibration
\[
K(R)\to K(\Gamma R)\to BE(\Sigma R)^+
\]
By \ref{coro:ksumtriv}, the map
\begin{equation}\label{map:alverre}
\Omega BE(\Sigma R)^+\to K(R)
\end{equation}
is a homotopy equivalence. Finally, by \ref{facts:bg+} ii),
\begin{equation}\label{map:aldere}
\Omega BE(\Sigma R)^+\weq \Omega_0K(\Sigma R).
\end{equation}
Now compose \eqref{map:aldere} with a homotopy inverse of \eqref{map:alverre} to obtain the theorem.\qed
\end{proof}

\begin{coro}
For all $n\in \Z$, $K_n(\Sigma R)=K_{n-1}(R)$
\end{coro}
\begin{proof}
For $n\le 0$, the statement of the proposition is immediate from the definition of $K_n$. For $n=1$, it is \eqref{decal}.
If $n\ge 2$, then
\[
K_n(\Sigma R)=\pi_n(K(\Sigma R))=\pi_{n-1}(\Omega K(\Sigma R))=\pi_{n-1}(\Omega_0 K(\Sigma R))=\pi_{n-1}K(R)=K_{n-1}R.\ \ \qed
\]
\end{proof}
\begin{rem}
The homotopy equivalence of Proposition \ref{prop:conseq} is the basis for the construction of the nonconnective $K$-theory
spectrum; we will come back to this in Section \ref{sec:spectra}.
\end{rem}
\subsection{Functoriality issues.}\label{subsec:functissues}
As defined, the rule $K:R\mapsto BGL(R)^+$ is only functorial up to homotopy. Actually its possible to choose  a
functorial model for $K R$; this can be done in different ways (see for example \cite[11.2.4,11.2.11]{lod}).
However, in the constructions and arguments we have made (notably in the proof of \ref{prop:conseq}) we have often
used Whitehead's theorem that a map of $CW$-complexes which induces an isomorphism at the level of homotopy groups
(a {\it weak equivalence}) always has a homotopy inverse. Now, there is in principle no reason why a natural weak
equivalence between functorial $CW$-complexes will admit a homotopy inverse which is also natural; thus for example,
the weak equivalence of Proposition \ref{prop:conseq} need not be natural for an arbitrarily chosen functorial
version of $K R$. What we need is to be able to choose functorial models so that any natural weak equivalence
between them automatically has a natural homotopy inverse. In fact we can actually do this, as we shall now see.
First of all, as a technical restriction we have to choose a small full subcategory $I$ of the category $\ass$, and
look at $K$-theory as a functor on $I$. This is no real restriction, as in practice we always start with a set of
rings (ofter with only one element) and then all the arguments and constructions we perform take place in a set
(possibly larger than the one we started with, but still a set). Next we invoke the fact that the category
$\topo_*^I$ of functors from $I$ to pointed spaces is a closed model category where fibrations and weak equivalences
are defined objectwise (by \cite[11.6.1]{hir} this is true of the category of functors to any cofibrantly generated
model category; by \cite[11.1.9]{hir}, $\topo_*$ is such a category). This implies, among other things, that there
is a full subcategory $(\topo_*^I)_c\to \topo_*^I$, the subcategory of cofibrant objects (among which any natural
weak equivalence has a natural homotopy inverse), a functor $\topo_*^I\to (\topo_*^I)_c$, $X\mapsto \hat{X}$ and a
natural transformation $\hat{X}\to X$ such that $\hat{X}(R)\to X(R)$ is a fibration and weak equivalence for all
$R$. Thus we can replace our given functorial model for $B\GL^+(R)$ by $\widehat{B\GL(R)^+}$, and redefine
$K(R)=\widehat{B\GL(R)^+}$.

\bn

\subsection{Relative $K$-groups and excision.}

Let $R$ be a unital ring, $I\triqui R$ an ideal, and $S=R/I$. Put
\[
\widebar{\GL}S:=\ima(\GL R\to \GL S)
\]
The inclusion $\widebar{\GL}S\subset\GL S$ induces a map
\begin{equation}\label{map:glbar}
(B\widebar{\GL}S)^+\to K(S)
\end{equation}
By \ref{facts:bg+} ii), \eqref{map:glbar} induces an isomorphism
\[
\pi_n(B\widebar{\GL}S)^+=K_nS \qquad (n\ge 2).
\]
On the other hand,
\[
\pi_1(B\widebar{\GL}S^+)=\widebar{\GL}S/ES=\ima(K_1R\to K_1S).
\]
Consider the homotopy fiber
\[
K(R:I):=\mathrm{fiber}((B\GL R)^+\to (B\widebar{\GL}S)^+).
\]
The {\it relative $K$-groups} of $I$ with respect to the ideal embedding $I\triqui R$ are defined
by
\[
K_n(R:I):=\begin{cases}\pi_n K(R:I) &n\ge 1\\ K_n(I) & n\le 0\end{cases}
\]
The long exact squence of homotopy groups of the fibration which defines $K(R:I)$, spliced together
 with the exact sequences of Theorem \ref{thm:exci01} and Proposition \ref{prop:kneg},
 yields a long exact sequence
\begin{equation}\label{relseq}
K_{n+1}R\to K_{n+1}S\to K_n(R:I)\to K_nR\to K_n(S)\qquad (n\in \Z)
\end{equation}
The canonical map $\tilde{I}\to R$ induces a map
\begin{equation}\label{abstorel}
K_n(I)\to K_n(R:I)
\end{equation}
This map is an isomorphism for $n\le 0$, but not in general (see Remark \ref{exa:swanex}). The rings $I$ so that this map is an isomorphism for all $n$ and $R$ are called {$K$-excisive}.
Suslin and Wodzicki have completely
characterized $K$-excisive rings (\cite{wodex},\cite{qs},\cite{sus}). We have
\begin{them}\label{thm:exito}(\cite{sus})
The map \eqref{abstorel} is an isomorphism for all $n$ and $R$
$\iff \tor^{\tilde{I}}_n(\Z,I)=0$ $\forall n$.\qed
\end{them}
Note that
\begin{gather*}
\tor_0^{\tilde{I}}(\Z,I)=I/I^2\\
\tor_n^{\tilde{I}}(\Z,I)=\tor_{n+1}^{\tilde{I}}(\Z,\Z).
\end{gather*}
\begin{exa}\label{exa:sigmanotexci}
Let $G$ be a group, $IG\triqui\Z G$ the augmentation ideal. Then $\Z G=\tilde{IG}$ is the unitalization of $IG$. Hence
\[
\tor_n^{\tilde{IG}}(\Z,IG)=H_{n+1}(G,\Z).
\]
In particular
\[
\tor_0^{\tilde{IG}}(\Z,I)=G_{ab}
\]
So if $IG$ is $K$-excisive, then $G$ must be a perfect group. Thus, for example, $IG$ is not $K$-excisive if $G$ is a nontrivial abelian group. In particular,
the ring $\sigma$ is not $K$-excisive, as it coincides with the augmentation ideal of  $\Z[\Z]=\Z[t,t^{-1}]$. As another example, if $S$ is an infinite sum ring, then
\[
H_n(\GL(S),\Z)=H_n(K(S),\Z)=H_n(pt,\Z)=0\qquad (n\ge 1).
\]
Thus the ring $I\GL(S)$ is $K$-excisive.
\end{exa}
\begin{rem}
We shall introduce a functorial complex $\bar{L}(A)$ which computes $\tor_*^{\tilde{A}}(\Z,A)$ and use
it to show that the functor $\tor_*^{\widetilde{(-)}}(\Z,-)$ commutes with filtering colimits. Consider
the functor $\perp:\tilde{A}-mod\to\tilde{A}-mod$,
\[
\perp M=\bigoplus_{m\in M}\tilde{A}.
\]
The functor $\perp$ is the free $\tilde{A}$-module cotriple \cite[8.6.6]{chubu}. Let $L(A)\to A$ be the canonical free resolution associated
to $\perp$ \cite[8.7.2]{chubu};
by definition, its $n$-th term is $L_n(A)=\perp^n A$. Put $\bar{L}(A)=\Z\otimes_{\tilde{A}}L(A)$.
Then $\bar{L}(A)$ is a functorial chain complex which satisfies
$H_*(\bar{L}(A))=\tor_*^{\tilde{A}}(\Z,A)$. Because $\perp$ commutes with filtering colimits, it follows that
the same is true of $L$ and $\bar{L}$, and therefore also of $\tor_*^{\widetilde{(-)}}(\Z,-)=H_*\bar{L}(-)$.
\end{rem}
\begin{exe}\label{exe:unitalexci}
\item{i)} Prove that any unital ring is $K$-excisive.
\item{ii)} Prove that if $R$ is a unital ring, then $M_\infty R$ is $K$-excisive. (Hint: $M_\infty R=\colim_nM_nR$).
\end{exe}

\begin{rem}
If $A$ is flat over $k$ (e.g. if $k$ is a field) then the canonical resolution $L^k(A)\weq A$ associated with the induced
module cotriple $\tilde{A}_k\otimes_k(-)$, is flat. Thus $\bar{L}^k(A):=L^k(A)/AL^k(A)$ computes $\tor_*^{\tilde{A}_k}(k,A)$.
Modding out by degenerates, we obtain
a homotopy equivalent complex (\cite{chubu}) $C^{\mbar}(A/k)$, with $C_n^{\mbar}(A/k)=A^{\otimes_k n+1}$. The complex
$C^{\mbar}$ is the bar complex considered by Wodzicki in \cite{wodex}; its homology is the bar homology of
$A$ relative to $k$, $H^{\mbar}_*(A/k)$. If $A$ is a $\Q$-algebra, then $A$ is flat as
a $\Z$-module, and thus $H^{\mbar}_*(A/\Z)=\tor^{\tilde{A}}_*(\Z,A)$. Moreover, as $A^{\otimes_\Z n}=A^{\otimes_\Q n}$, we have
$C^{\mbar}(A/\Z)=C^{\mbar}(A/\Q)$, whence
\[
\tor^{\tilde{A}}_*(\Z,A)=\tor^{\tilde{A}_\Q}_*(\Q,A)=H^{\mbar}_*(A/\Q)
\]
\end{rem}

\subsection{Locally convex algebras.}
A {\it locally convex} algebra is a complete topological $\C$-algebra $L$ with a locally convex topology. Such a topology is
defined by a family of seminorms $\{\rho_\alpha\}$; continuity of the product means that for every $\alpha$ there exists a $\beta$
such that
\begin{equation}\label{multcont}
\rho_\alpha(xy)\le \rho_\beta(x)\rho_\beta(y)\qquad (x,y\in L).
\end{equation}
If in addition the topology is determined by a countable family of seminorms, we say that $L$ is a {\it Fr\'echet} algebra.
\mn

Let $L$ be a locally convex algebra. Consider the following two factorization properties:
\begin{itemize}
\item[(a)] Cohen-Hewitt factorization.
\begin{gather}\label{ppty:F}
\qquad \forall n\ge 1,a=(a_1,\dots,a_n)\in L^{\oplus n}=\bigoplus_{i=1}^nL\quad\exists z\in L,\quad x\in L^n\text{ such
that}\\
z\cdot x=a\text{ and } x\in \widebar{L\cdot a}\nonumber
\end{gather}
Here the bar denotes topological closure in $L^{\oplus n}$.
\item[(b)] Triple factorization.
\begin{gather}\label{ppty:T}
\forall a\in L^{\oplus n}, \exists b\in L^{\oplus n}, \quad c,d \in L,\\ \text{such that}
a=cdb \text{ and } (0:d)_l:=\{v\in L:dv=0\}= (0:cd)_l
\end{gather}
\end{itemize}
The right ideal $(0:d)_l$ is called the {\it left annihilator} of $d$. Note that property (b) makes
sense for an arbitrary ring $L$.

\begin{lem}\label{lem:equifacto}
Cohen-Hewitt factorization implies triple factorization. That is, if $L$ is a locally convex algebra which
satisfies property (a) above, then it also satisfies property (b).
\end{lem}
\begin{proof}
Let $a\in L^{\oplus n}$. By (a), there exist $b\in L^{\oplus n}$ and $z\in L$ such that $a=zb$. Applying
(a) again, we get that $z=cd$ with $d\in \widebar{L\cdot z}$; this implies that $(0:d)_l=(0:z)_l$.
\end{proof}

\begin{them}\label{prop:exciloca}(\cite[3.12, 3.13(a)]{qs})
Let $L$ be a ring. Assume that either $L$ or $L^{op}$ satisfy \eqref{ppty:T}.
Then $A$ is $K$-excisive.\qed
\end{them}

\subsection{Fr\'echet $m$-algebras with approximate units.}

A  {\it uniformly bounded left approximate unit} (ublau) in a locally convex algebra $L$ is a net $\{e_\lambda\}$ of
elements of $L$ such that $e_\lambda a\mapsto a$ for all $a$ and $\sup_\alpha\rho_\alpha(a)<\infty$. Right ubau's are defined analogously.
If $L$ is a locally convex algebra such that a defining family of seminorms can be chosen so that condition \eqref{multcont}
is satisfied with $\alpha=\beta$ (i.e. the seminorms are {\it submultiplicative}) we say that $L$ is an $m$-algebra.
An $m$-algebra which is also Fr\'echet will be called a {\it Fr\'echet $m$-algebra}.
\begin{exa}\label{exa:ubau} Every $C^*$-algebra has a two-sided ubau (\cite[I.4.8]{david}). If $G$ is a locally compact group, then the group algebra $L^1(G)$ is a Banach algebra
with two sided ubau \cite[8.4]{wodex}. If $L_1$ and $L_2$ are locally convex algebras with ublaus $\{e_\lambda\}$ and $\{f_\mu\}$, then
$\{e_\lambda\otimes f_\mu\}$ is a ublau for the projective tensor product $L_1\hotimes L_2$, which is a (Fr\'echet) $m$-algebra
if both $L_1$ and $L_2$ are.
\end{exa}

\begin{rem}
In a Banach algebra,
any bounded approximate unit is uniformly bounded. Thus for example, the unit of a unital Banach algebra
is an ublau. However, the unit of a general unital locally convex algebra (or even of a Fr\'echet $m$-algebra) need not
be uniformly bounded.
\end{rem}
Let $L$ be an $m$-Fr\'echet algebra. A left {\it Fr\'echet $L$-module} is a Fr\'echet space $V$ equipped with a left
$L$-module structure such that the multiplication map $L\times V\to V$ is continuous.
If $L$ is equipped with an ublau $e_\lambda$ such that $e_\lambda\cdot v\to v$ for all $v\in V$, then we say that $V$ is
{\it essential}.
\begin{exa}\label{exa:summ}
If $L$ is an $m$-Fr\'echet algebra with ublau $e_\lambda$ and $x\in L^{\oplus n}$ $(n\ge 1)$ then $e_\lambda x\to x$. Thus $L^{\oplus n}$
is an essential Fr\'echet $L$-module. The next exercise
generalizes this example.
\end{exa}
\begin{exe}\label{exe:summ}
Let $L$ be an $m$-Fr\'echet algebra with ublau $e_\lambda$, $M$ a unital $m$-Fr\'echet algebra, and $n\ge 1$. Prove that for
every $x\in (L\hotimes M)^{\oplus n}$, $(e_\lambda\otimes 1)x\to x$. Conclude that $(L\hotimes M)^{\oplus n}$ is an essential
$L\hotimes M$-module.
\end{exe}

The following Fr\'echet version of Cohen-Hewitt's factorization theorem (originally proved in the Banach
setting) is due to M. Summers.
\begin{them}\label{thm:summ}(\cite[2.1]{summ}) Let $L$ be an $m$-Fr\'echet algebra with ublau,
and $V$ an essential Fr\'echet left $L$-module. Then for each $v\in V$ and for each neighbourhood $U$ of the origin in $V$
there is an $a\in L$ and a $w\in V$ such that $v=aw$, $w\in\overline{Lv}$, and $w-v\in U$.\qed
\end{them}

\begin{them}\label{thm:ubau}(\cite[8.1]{wodex})
Let $L$ be a Fr\'echet $m$-algebra. Assume $L$ has a right or left ubau. Then $L$ is $K$-excisive.
\end{them}
\begin{proof}
In view of Lemma \eqref{lem:equifacto}, it suffices to show that $L$ satisfies property \eqref{ppty:F}. This follows by applying Theorem \ref{thm:summ}
to the essential $L$-module $L^{\oplus n}$.\qed
\end{proof}

\begin{exe}\label{exe:summ2}
Prove that if $L$ and $M$ are as in Exercise \eqref{exe:summ}, then $L\hotimes M$ is $K$-excisive.
\end{exe}

\begin{rem}
In \cite[8.1.1]{cot} it asserted that if $k\supset \Q$ is a field, and $A$ is a $k$-algebra, then
$\tor^{\tilde{A}_\Q}_*(\Q,A)=\tor^{\tilde{A}_k}_*(k,A)$, but the proof uses the identity
$\tilde{A}_k\otimes_{\tilde{A}}{\ ? \ }=k\otimes{\ ? \ }$, which is wrong. In {\it loc. cit.}, the lemma is used in combination
with Wodzicki's theorem (\cite[8.1]{wodex}) that a Fr\'echet algebra $L$ with ublau is $H$-unital
as a $\C$-algebra, to conclude that such $L$ is $K$-excisive. In Theorem \ref{thm:ubau} we gave
a different proof of the latter fact.
\end{rem}

\subsection{Fundamental theorem and the Toeplitz ring.}\label{subsec:fundatoep}
\begin{notation} If $G:\ass\to\ab$ is a functor, and $A$ is a ring, we put
\[
NG(A):=\coker(GA\to G(A[t])).
\]
\end{notation}

Let $R$ be a unital ring. We have a commutative diagram
\[
\xymatrix{R\ar[d]\ar[r]&R[t]\ar[d]\\ R[t^{-1}]\ar[r]&R[t,t^{-1}]}
\]
Thus applying the functor $K_n$ we obtain a map
\begin{equation}\label{map:fund1}
K_nR\oplus NK_nR\oplus NK_nR\to K_n R[t,t^{-1}]
\end{equation}
which sends $NK_nR\oplus NK_nR$ inside $\ker\ev_1$. Thus $K_nR\to
K_nR[t,t^{-1}]$ is a split mono, and the intersection of its image
with that of $NK_nR\oplus NK_nR$ is $0$. On the other hand, the
inclusion $\cT R\to \Gamma R$ induces a map of exact sequences
\[
\xymatrix{0\ar[r]&M_\infty R\ar@{=}[d]\ar[r]&\cT R\ar[d]\ar[r]&R[t,t^{-1}]\ar[d]\ar[r]& 0\\
0\ar[r]&M_\infty R\ar[r]&\Gamma R\ar[r]&\Sigma R\ar[r]& 0}
\]
In particular, we have a homomorphism $R[t,t^{-1}]\to\Sigma R$, and thus a homomorphism
\[
\eta:K_nR[t,t^{-1}]\to K_{n-1} R.
\]
Note that the maps $R[t]\to \cT R$, $t\mapsto \alpha$ and
$t\mapsto \alpha^*$, lift the homomorphisms $R[t]\to R[t,t^{-1}]$,
$t\mapsto t$ and $t\mapsto t^{-1}$. It follows that $\ker\eta$
contains the image of \eqref{map:fund1}. In \cite{lorep}, Loday
introduced a product operation in $K$-theory of unital rings
\[
K_p(R)\otimes K_q(S)\to K_{p+q}(R\otimes S).
\]
In particular, multiplying by the class of $t\in K_1(\Z[t,t^{-1}])$ induces a map
\begin{equation}\label{map:fund2}
\cup t:K_{n-1}R\to K_nR[t,t^{-1}].
\end{equation}
Loday proves in \cite[2.3.5]{lorep} that $\eta\circ (-\cup t)$ is
the identity map. Thus the images of \eqref{map:fund1} and
\eqref{map:fund2} have zero intersection. Moreover, we have the following result, due to Quillen
\cite{gray}, which is known as the fundamental theorem of $K$-theory.

\begin{them}\label{thm:fundak}(\cite{gray}, see also \cite{sch})
Let $R$ be a unital ring. The maps \eqref{map:fund1} and \eqref{map:fund2} induce an isomorphism
\[
K_nR\oplus NK_nR\oplus NK_nR\oplus K_{n-1}R\iso
K_nR[t,t^{-1}]\qquad (n\in\Z).
\]
\end{them}
\begin{coro}\label{cor:fundak} (cf. Theorem \ref{thm:decalkh})
\[
K_n(R[t,t^{-1}]:\sigma R)=K_{n-1}R\oplus NK_nR\oplus
NK_nR\qquad(n\in\Z).
\]
\end{coro}
\begin{propo}\label{prop:ktoep}(cf. Theorem \ref{thm:algcuntz})
Let $R$ be a unital ring, and $n\in\Z$. Then
\begin{gather*}
K_n\cT R=K_nR\oplus NK_nR\oplus NK_nR,\\
K_n(\cT R:\cT_0R)=NK_nR\oplus NK_nR.
\end{gather*}
\end{propo}
\begin{proof}
Consider the exact sequence
\[
0\to M_\infty R\to \cT R\to R[t,t^{-1}]\to 0
\]
By Proposition \ref{prop:exciexci}, Exercise \ref{exe:unitalexci} and matrix stability we have a long exact sequence
\[
K_n R\to K_n\cT R\to K_n R[t,t^{-1}]\to K_{n-1} R\to K_{n-1}\cT R\qquad (n\in\Z).
\]
By \ref{prop:atotau=0}, the first and the last map are zero. The proposition is
immediate from this, from Corollary \ref{cor:fundak}, and from the
discussion above.\qed
\end{proof}
\section{Comparison between algebraic and topological $K$-theory I}\label{sec:compa1}

\subsection{Stable $C^*$-algebras.}

The following is Higson's homotopy invariance theorem.

\begin{them}\label{thm:hit1}\cite[3.2.2]{hig}
Let $G$ be a functor from $C^*$-algebras to abelian groups. Assume that $G$ is split exact and $\cK$-stable. Then
$G$ is homotopy invariant.\qed
\end{them}

\begin{lem}\label{lem:stab} Let $G$ be a functor from $C^*$-algebras to abelian groups. Assume that $G$ is $M_2$-stable. Then
the functor $F(A):=G(A\sotimes \cK)$ is $\cK$-stable.
\end{lem}
\begin{proof} Let $H$ be an infinite dimensional separable Hilbert space. The canonical isomorphism
$\C^2\otimes_2 H\cong H\oplus H$ induces an isomorphism $\cK\sotimes \cK\to M_2\cK$ which makes the following diagram
commute
\[
\xymatrix{\cK\sotimes\cK\ar[rr]^{\cong}&& M_2\cK\\
                 &\cK\ar[ul]^{e_{11}\otimes 1}\ar[ur]^{\iota_1}&}
\]
Since $G$ is $M_2$-stable by hypothesis, it follows that $F(1_A\sotimes e_{1,1}\sotimes 1_\cK)$ is an isomorphism for
all $A$.\qed
\end{proof}
The following result, due to Suslin and Wodzicki, is (one of the variants of) what is known as Karoubi's conjecture \cite{karcomp}.


\begin{them}\label{thm:karconc*}\cite[10.9]{qs}
Let $A$ be a $C^*$-algebra. Then there is a natural isomorphism $K_n(A\sotimes \cK)=K^{\top}_n(A\sotimes\cK)$.
\end{them}
\begin{proof}
By definition $K_0=K_0^{\top}$ on all $C^*$-algebras. By Example \ref{exa:ubau} and Theorem \ref{thm:ubau},
$C^*$-algebras are $K$-excisive. In particular $K_*$ is split exact when regarded as a functor of $C^*$-algebras.
By \ref{prop:hikay} iii), \ref{prop:kneg} i), and split exactness, $K_*$ is $M_\infty$-stable on $C^*$-algebras; this
implies it is also $M_2$-stable (Exercise \ref{exe:matrix}). Thus $K_*(-\sotimes\cK)$ is $\cK$-stable, by \ref{lem:stab}.
Hence  $K_n(A(0,1]\sotimes\cK)=0$, by split exactness and homotopy invariance (Theorem \ref{thm:hit1}). It follows
that
\begin{equation}\label{downshift}
K_{n+1}(A\sotimes\cK)=K_n(A(0,1)\sotimes \cK)
\end{equation}
by excision. In particular, for $n\ge 0$,
\begin{equation}\label{agreepos}
K_n(A\sotimes\cK)=K_0(A\sotimes \sotimes_{i=1}^n\C(0,1)\sotimes\cK)=K_n^{\top}(A\sotimes\cK).
\end{equation}
On the other hand, by Cuntz' theorem \ref{thm:cu}, excision applied to the $C^*$-Toeplitz extension and \ref{lem:stab},
$K_{n+1}(A(0,1)\sotimes\cK)=K_n(A\sotimes\cK\sotimes\cK)=K_n(A\sotimes\cK)$.
Putting this together with \eqref{downshift}, we get that $K_*(\cK\sotimes A)$ is Bott periodic. It follows that the
identity \eqref{agreepos} holds for all $n\in \Z$.\qed
\end{proof}

\subsection{Stable Banach algebras.}
The following result is a particular case of a theorem of Wodzicki.

\begin{them}\label{thm:karconbau}(\cite[Thm. 2]{wod}, \cite[8.3.3, 8.3.4]{cot})
Let $L$ be Banach algebra with right or left ubau. Then there is an isomorphism $K_*(L\hotimes\cK)=K_*^{\top}(L\hotimes\cK)$.
\end{them}
\begin{proof}
Consider the functor $G_L:C^*\to\ab$, $A\mapsto K_*(L\hotimes(A\sotimes\cK))$. By the same argument as in the proof
of \ref{thm:karconc*}, $G_L$ is homotopy invariant.
Hence $\C\to \C[0,1]$ induces an isomorphism
\begin{align*}
G_L(\C)=&K_*(L\hotimes \cK)\iso G_L(\C[0,1])=K_*(L\hotimes(\C[0,1]\sotimes\cK))\\
=&K_*(L\hotimes \cK[0,1])=K_*((L\hotimes\cK)[0,1]).
\end{align*}
Hence $K_{n+1}(L\hotimes\cK)=K_n(L\hotimes\cK(0,1))$, by \ref{thm:ubau} and \ref{exa:ubau}. Thus
$K_n(L\hotimes\cK)=K_n^{\top}(L\hotimes\cK)$ for $n\ge 0$. Consider the punctured Toeplitz sequence
\[
0\to \cK\to \cT^{\top}_0\to \C(0,1)\to 0
\]
By \cite[V.1.5]{david}, this sequence admits a continuous linear splitting. Hence it remains exact after applying the functor
$L\hotimes-$. By \ref{thm:cu}, we have
\[
K_{-n}(L\hotimes(\cK\sotimes\cT_0))=0\qquad (n\ge 0).
\]
Thus
\[
K_{-n}(L\hotimes\cK)=K_{-n}(L\hotimes(\cK\sotimes\cK))=K_0((L\hotimes\cK)\hotimes\hotimes_{i=1}^n\C(0,1))=K^{\top}_{-n}(L\hotimes \cK).\ \ \qed
\]
\end{proof}
\begin{rem}
The theorem above holds more generally for $m$-Fr\'echet algebras (\cite[Thm. 2]{wod}, \cite[8.3.4]{cot}), with the appropriate definition
of topological $K$-theory (see Section \ref{sec:karconfre} below).
\end{rem}

\begin{exe}
Let $A$ be a Banach algebra. Consider the map $K_0(A)\to K_{-1}(A(0,1))$ coming from the exact sequence
\[
0\to A(0,1)\to A(0,1]\to A\to 0
\]
Put
\[
A(0,1)^m=A\hotimes\left(\hotimes_{i=1}^{m}\C(0,1)\right)
\]
and define
\[
KC_n(A)=\colim_p K_{-p}(A(0,1)^{n+p})\qquad (n\in \Z)
\]
\item{i)} Prove that $KC_*$ satisfies excision, $M_\infty$-stability, continuous homotopy invariance, and nilinvariance.
\item{ii)} Prove that $KC_*(A\hotimes\cK)=K_*^{\top}(A)$.
\item{iii)} Prove that the composite
\[K_n^{\top}A=K_0(A(0,1)^n)\to KC_n(A)\to KC_n(A\hotimes\cK)=K^{\top}_n(A)\qquad( n\ge 0)\]

is the identity map. In particular $KC_n(A)\to K^{\top}_n(A)$ is surjective for $n\ge 0$.
\end{exe}
\begin{rem}
J. Rosenberg has conjectured (see \cite[3.7]{rosurvey}) that, for $n\le -1$, the restriction of $K_n$ to
commutative $C^*$-algebras is homotopy invariant. Note that if $A$ is a Banach algebra (commutative or not)
such that, $K_{-q}(A(0,1)^p)\to K_{-q}(A(0,1)^p[0,1])$ is an isomorphism for all $p,q\ge 0$, then
$KC_n(A)\to K^{\top}_nA$ is an isomorphism for all $n$. In particular, if Rosenberg's conjecture holds, this
will happen for all commutative $C^*$-algebras $A$.

\end{rem}
\mn

\section{Topological $K$-theory for locally convex algebras}\label{sec:topk2}

\subsection{Diffeotopy $KV$.}
We begin by recalling the notion of $C^\infty$-homotopies or diffeotopies (from \cite{cd}, \cite{cw}).
Let $L$ be a locally convex algebra. Write $\cC^\infty([0,1],L)$ for  the algebra of those functions $[0,1]\to L$ which are restrictions
of $\cC^\infty$-functions $\R\to L$. The algebra $\cC^\infty([0,1],L)$ is equipped with a locally convex topology which makes
it into a locally convex algebra, and there is a canonical isomorphism 
\[
\cC^\infty([0,1],L)=\cC^\infty([0,1],\C)\hotimes L
\]
Two homomorphisms $f_0,f_1:L\to M$ of locally convex algebras are called {\it diffeotopic} if there is a homomorphism $H:L\to \cC^\infty([0,1],M)$
such that the following diagram commutes
\[
\xymatrix{&\cC^\infty([0,1],M)\ar[d]^{(\ev_0,\ev_1)}\\ L\ar[ur]^H\ar[r]_{(f_0,f_1)}&M\times M}
\]
Consider the exact sequences
\begin{gather}\label{lopecito}
0\to P^{\dif}L\to \cC^\infty([0,1],L)\overset{\ev_0}\to L\to 0\\
0\to \Omega^{\dif}L\to P^{\dif}L\overset{\ev_1}\to L\to 0\label{gonzalez}
\end{gather}
Here $P^{\dif}L$ and $\Omega^{\dif}L$ are the kernels of the evaluation maps. The first of these is split by the natural inclusion
$L\to \cC^\infty([0,1],L)$, and the second is split the continous linear map sending $l\mapsto (t\mapsto tl)$. We have
\[
\Omega^{\dif}L=\Omega^{\dif}\C\hotimes L,\qquad P^{\dif}L=P^{\dif}\C\hotimes L.
\]

Put
\begin{gather*}
\GL(L)_0^{\prime\prime}=\ima(\GL P^{\dif} L\to \GL (L))\\
KV^{\dif}_1(L)=\GL(L)/\GL(L)_0^{\prime\prime}.
\end{gather*}
The following is the analogue of Proposition \ref{prop:kv1ppties} for $KV_1^{\dif}$ (except for nilinvariance,
treated separately in Exercise \ref{exe:nilkvdif}).

\begin{propo}\label{prop:kvdppties}
\item{i)} The functor $KV^{\dif}_1$ is split exact.
\item{ii)}For each locally convex algebra $L$, there is a natural surjective map $K_1L\to KV^{\dif}_1L$.
\item{iii)} If
\begin{equation}\label{seq:lmn}
0\to L\to M\to N\to 0
\end{equation}
is an exact sequence such that the map $\GL(M)^{\prime\prime}_0\to \GL(N)^{\prime\prime}_0$ is onto, then
the map $K_1N\to K_0L$ of Theorem \ref{thm:exci01} factors through $KV^{\dif}_1N$, and the resulting sequence
\[
\xymatrix{KV^{\dif}_1L\ar[r] &KV^{\dif}_1M\ar[r]& KV^{\dif}_1N\ar[d]^{\partial}\\K_0N&K_0M\ar[l]&K_0L\ar[l]}
\]
is exact.
\item{iv)}$KV^{\dif}_1$ is additive, diffeotopy invariant and $M_\infty$-stable.
\end{propo}
\begin{proof} One checks that, mutatis-mutandis, the same argument of the proof of \ref{prop:kv1ppties} shows this.\qed
\end{proof}
By the same argument as in the algebraic case, we obtain a natural injection
\[
KV_1^{\dif}L\hookrightarrow K_0(\Omega^{\dif}L)
\]
Higher $KV^{\dif}$-groups are defined by
\[
KV^{\dif}_n(L)=KV_1^{\dif}((\Omega^{\dif})^{n-1}L)\qquad (n\ge 2)
\]
\begin{exe}\label{exe:nilkvdif}
\item{i)} Show that if $L$ is a locally convex algebra such that $L^n=0$ and such that $L\to L/L^i$
admits a continuous linear splitting for all $i\le n-1$, then $KV^{\dif}_1L=0$.
\item{ii)} Show that if $L$ is as in i) then the map $KV^{\dif}_1M\to KV^{\dif}_1N$ induced by \eqref{seq:lmn}
is an isomorphism.
\end{exe}

\subsection{Diffeotopy $K$-theory.}

Consider the excision map
\[
K_nL\to K_{n-1}(\Omega^{\dif}L)\qquad(n\le 0)
\]
associated to the sequence \eqref{gonzalez}. The {\it diffeotopy $K$-theory} of the algebra $L$ is defined
by the formula
\[
KD_nL=\colim_pK_{-p}((\Omega^{\dif})^{n+p}L)\qquad (n\in\Z)
\]
It is also possible to express $KD$ in terms of $KV^{\dif}$. First we observe that, since $\Sigma \C$
is a countably dimensional algebra, equipping it with the fine topology makes it into a locally convex algebra \cite[2.1]{cw}, and if $L$ is
any locally convex algebra then we have
\[
\Sigma L=\Sigma\C\otimes_\C L=\Sigma\C\hotimes L.
\]
Thus
\[
\Omega^{\dif}\Sigma L=\Sigma\Omega^{\dif}L.
\]
Taking this into account, and using the same argument as used to prove \eqref{khkv}, one obtains
\[
KD_nL=\colim_rKV^{\dif}_1(\Sigma^{r+1}(\Omega^{\dif})^{n+r}L)=\colim_rKV^{\dif}_{n+r+1}(\Sigma^{r+1}L).
\]

\begin{propo}\label{prop:kdppties}
Diffeotopy $K$-theory has the following properties.

\item[i)] It is diffeotopy invariant, nilinvariant and $M_\infty$-stable.
\item[ii)] It satisfies excision for those exact sequences which admit a continuous linear splitting.
That is, if
\begin{equation}\label{seq:lsplit}
0\to L\to M\overset{\pi}\to N\to 0
\end{equation}
is an exact sequence of locally convex algebras and there exists a continuous linear map $s:N\to M$ such that
$\pi s=1_N$, then there is a long exact sequence
\[
KD_{n+1}M\to KD_{n+1}N\to KD_nL\to KD_nM\to KD_nN\qquad (n\in\Z).
\]
\end{propo}
\begin{proof} The proof is essentially the same as that of Theorem \eqref{thm:khppties}. The splitting hypothesis in ii) guarantees
that the functor $L\mapsto \Omega^{\dif}L=\Omega^{\dif}\C\hotimes L$ and its iterations, send \eqref{seq:lsplit} to an exact sequence.\qed
\end{proof}

\paragraph{Comparing $KV^{\dif}$ and $KD$.}

The analogue of Proposition \ref{prop:khk0reg} is \ref{prop:kdk0reg} below. It is immediate from Lemma \ref{lem:plreg}, which is the analogue of
Lemma \ref{lem:pareg};
the proof of \ref{lem:plreg} is essentially the same as that of \ref{lem:pareg}.

\begin{lem}\label{lem:plreg}
Let $L$ be a locally convex algebra. Assume that for all $n\le 0$ and all $p\ge 1$, the natural inclusion
$\iota_p:L\to L\hotimes\left(\hotimes_{i=1}^p\cC^\infty([0,1])\right)=\cC^\infty([0,1]^p,L)$ induces an isomorphism
$K_n(L)\iso K_n\cC^\infty([0,1]^p,L)$. Then $KV_1^{\dif}L\to K_0\Omega^{\dif}L$ is an isomorphism, and for every
$n\le 0$ and every $p\ge 0$, $K_n(\cC^\infty([0,1]^p,P^{\dif}A))=0$
and $K_n(\Omega^{\dif}A)\to K_n(C^\infty([0,1]^p,\Omega^{\dif}A))$ is an isomorphism.
\end{lem}
\begin{propo}\label{prop:kdk0reg}
Let $L$ be a locally convex algebra. Assume $L$ satisfies the hypothesis of Lemma \ref{lem:plreg}. Then
\[
KD_nL=\begin{cases} KV^{\dif}_nL&n\ge 1\\ K_nL&n\le 0\end{cases}
\]
\end{propo}

\subsection{Bott periodicity.}
Next we are going to prove a version of Bott periodicity for $KD$. The proof is analogous to Cuntz' proof of Bott periodicity
for $K^{\top}$ of $C^*$-algebras, with the algebra of smooth compact operators and the smooth Toeplitz algebra substituted for the
$C^*$-algebra of compact operators and the Toeplitz $C^*$-algebra.

\paragraph{Smooth compact operators.}
The algebra $\fk$ of {\it smooth compact operators} (\cite[\S2]{ncp},\cite[1.4]{cd}) consists of all those $\N\times \N$-matrices $(z_{i,j})$ with complex coefficients such
that for all $n$,
\[
\rho_n(z):=\sum_{p,q}p^nq^n|z_{p,q}|<\infty
\]
The seminorms $\rho_n$ are submultiplicative, and define a locally convex topology on $\fk$. Since the topology is defined by submultiplicative
seminorms, it is an $m$-algebra. Further because the seminorms above are countably many, it
is Fr\'echet; summing up $\fk$ is an $m$-F\'echet algebra. We have a map
\[
e_{11}:\C\to \com, z\mapsto e_{11}z
\]
Whenever we refer to $\com$-stability below, we shall mean stability with respect to the functor $\com\hotimes-$ and the map $e_{11}$.

\paragraph{Smooth Toeplitz algebra.}
The {\it smooth Toeplitz algebra} (\cite[1.5]{cd}), is the free $m$-algebra $\cT^{\sm}$ on two generators $\alpha$, $\alpha^*$ subject to $\alpha\alpha^*=1$. As in the $C^*$-algebra case, there is a commutative diagram with exact rows and split exact columns
\[
\xymatrix{0\ar[r]&\com\ar@{=}[d]\ar[r]&\cT^{\sm}_0\ar[r]\ar[d]&\Omega^{\dif}\C\ar[r]\ar[d]&0\\
          0\ar[r]&\com\ar[r]&\cT^{\sm}\ar[r]\ar[d]&\cC^\infty(S^1,\C)\ar[r]\ar[d]_{\ev_1}&0\\
&&\C\ar@{=}[r]&\C&}
\]
Here $\cT^{\sm}_0$ is defined so that the middle column be exact, and we use the exponential map
to identify $\Omega^{\dif}\C$ with the kernel of the evaluation map. Moreover the construction of $\cT^{\sm}$ given in \cite{cd} makes it clear that the rows are exact with a continuous linear splitting, and thus
they remain exact after applying $L\hotimes $, where $L$ is any locally convex algebra.

\paragraph{Bott periodicity.}

The following theorem, due to J. Cuntz, appears in \cite[Satz 6.4]{cd}, where it is stated for functors on $m$-locally convex algebras. The same
proof works for functors of all locally convex algebras.

\begin{them}\label{thm:cusm}(Cuntz, \cite[Satz 6.4]{cd})
Let $G$ be a functor from locally convex algebras to abelian groups. Assume that
\begin{itemize}
\item $G$ is diffeotopy invariant.
\item $G$ is $\com$-stable.
\item $G$ is split exact.
\end{itemize}
Then for every locally convex algebra $L$, we have:
\[
G(L\hotimes\cT_0^{\sm})=0
\]
\end{them}

\begin{them}\label{thm:kdbott}
For every locally convex algebra $L$, there is a natural isomorphism $KD_*(L\hotimes\com)\cong KD_{*+2}(L\hotimes\com)$.
\end{them}
\begin{proof}
Consider the exact sequence
\[
\xymatrix{0\ar[r]&L\hotimes\com\ar[r]&L\hotimes\cT_0^{\sm}\ar[r]&\Omega^{\dif}L\ar[r]&0}
\]
This sequence
is linearly split by construction (see \cite[1.5]{cd}). This splitting property is clearly preserved if we apply the functor $\com\hotimes$. Hence by Proposition \ref{prop:kdppties} ii), we have a natural map
 \begin{gather}\label{edgeloop}
 KD_{*+1}(L\hotimes\com)=KD_*(\Omega^{\dif}L\hotimes\com)\to KD_{*-1}(L\hotimes\com\hotimes\com)
 \end{gather}
By \cite[Lemma 1.4.1]{cd}, the map $1\hotimes e_{11}:\com\to \com\hotimes\com$ is diffeotopic to an isomorphism. Since $KD$ is diffeotopy
invariant, this shows that $KD_*(\com\hotimes-)$ is $\com$-stable. Hence $KD_{*-1}(L\hotimes\com\hotimes\com)=KD_{*-1}(L\hotimes\com)$, and
by Cuntz' theorem \ref{thm:cusm}, \eqref{edgeloop} is an isomorphism.\qed
\end{proof}
\begin{rem} Cuntz has defined a bivariant topological $K$-theory for locally convex algebras (\cite{cw}).
This theory associates groups $kk_*^{\lc}(L,M)$ to any pair $(L,M)$ of locally convex algebras, and is contravariant
in the first variable and covariant in the second. Roughly speaking, $kk_n^{\lc}(L,M)$ is defined as a certain
colimit of diffeotopy classes of $m$-fold extensions of $L$ by $M$ $(m\ge n)$. There is also an algebraic version of
Cuntz' theory, $kk_*(A,B)$, which is defined for all pairs of rings $(A,B)$ (\cite{biva}). We point out that
\begin{equation}\label{kkkd}
kk^{lc}_*(\C,M)=KD_*(M\hotimes\com).
\end{equation}
Indeed the proof given in \cite[8.1.2]{biva} that
for algebraic $kk$, $KH_*(A)=kk_*(\Z,A)$ for all rings $A$, can be adapted to prove \eqref{kkkd}; one
just needs to observe that, for the algebraic suspension, $kk^{\lc}_*(L,\Sigma M)=kk^{\lc}_{*-1}(L,M)$.
Note that, in view of the definition of $KD$, \eqref{kkkd} implies the following ``algebraic" formula for $kk^{\lc}$:
\[
kk^{\lc}_n(\C,L)=\colim_pK_{-p}(({\Omega^{\dif}})^{p+q}(L\otimes\com)).
\]
\end{rem}
\section{Comparison between algebraic and topological $K$-theory II}\label{sec:compa2}

\subsection{The diffeotopy invariance theorem.}
Let $H$ be an infinite dimensional separable Hilbert space; write $H\otimes_2H$ for the completed tensor product of Hilbert spaces. Note any two
infinite dimensional separable Hilbert spaces are isomorphic; hence we may regard any operator ideal $\J\triqui \cB(H)$ as a functor from
Hilbert spaces to $\C$-algebras (see \cite[3.3]{hus}). Let
$\J\triqui \cB$ be an ideal.
\begin{itemize}
\item $\J$ is {\it multiplicative} if $\cB\hotimes\cB\to \cB(H\otimes_2 H)$ maps $\J\hotimes \J$ to $\J$.
\item $\J$ is {\it Fr\'echet} if it is a Fr\'echet algebra and the inclusion $\J\to \cB$ is continuous.
A Fr\'echet ideal is a {\it Banach} ideal if it is a Banach algebra.
\sn
Write $\omega=(1/n)_n$ for the harmonic sequence.
\sn
\item $\J$ is {\it harmonic} if it is a multiplicative Banach ideal such that $\J(\ell^2(\N))$ contains $diag(\omega)$.
\end{itemize}
\begin{exa}
Let $p\in \R_{>0}$. Write $\cL_p$ for the ideal of those compact operators whose sequence of singular values is $p$-summable; $\cL_p$
is called the $p$-{\it Schatten ideal}. It is Banach $\iff$ $p\ge 1$, and is harmonic $\iff$ $p>1$. There is no interesting locally convex topology on $\cL_p$ for
$p<1$.
\end{exa}
\bn

The following theorem, due to J. Cuntz and A. Thom, is the analogue of Higson's homotopy invariance theorem \ref{thm:hit1} in the locally convex algebra
context. The formulation
we use here is a consequence of \cite[5.1.2]{ct} and \cite[4.2.1]{ct}.
\begin{them}\label{thm:hit2}(\cite{ct})
Let $\J$ be a harmonic operator ideal, and $G$ a functor from locally convex algebras to abelian groups. Assume that
\begin{itemize}
\item[i)] $G$ is $M_2$-stable.
\item[ii)] $G$ is split exact.
\end{itemize}
Then $L\mapsto G(L\hotimes \J)$ is diffeotopy invariant.
\end{them}
We shall need a variant of \ref{thm:hit2} which is valid for all Fr\'echet ideals $\J$. In order to state it,
we introduce some notation.
Let $\alpha:L\to M$
be a homomorphism of locally convex algebras. We say that $\alpha$ is an {\it isomorphism up to square zero}
if there exists a continous linear map $\beta:M\hotimes M\to L$ such that the compositions $\beta\circ(\alpha\hotimes\alpha)$
and $\alpha\circ\beta$ are the multiplication maps of $L$ and $M$. Note that if $\alpha$ is an isomorphism up to square zero,
then its image is a ideal of $M$, and both its kernel and its cokernel are square-zero algebras.

\begin{defi}\label{defi:nilinv} Let $G$ be a functor from locally convex algebras to abelian groups. We call $G$
{\rm continously nilinvariant} if it sends isomorphisms
up to square zero into isomorphisms.
\end{defi}
\begin{exa}\label{exa:khkosher}
For any $n\in\Z$, $KH_n$ is a continously nilinvariant functor
of locally convex algebras. If $n\le 0$, the same is true of $K_n$. In general,
if $H_*$ is the restriction to locally convex algebras of any excisive, nilinvariant homology theory of rings,
then $H_*$ is continously nilinvariant.
\end{exa}
\begin{them}\label{thm:hit3}\cite[6.1.6]{cot}
Let $\J$ be a Fr\'echet operator ideal, and $G$ a functor from locally convex algebras to abelian groups. Assume that
\begin{itemize}
\item[i)] $G$ is $M_2$-stable.
\item[ii)] $G$ is split exact.
\item[iii)] $G$ is continuously nilinvariant.
\end{itemize}
Then $L\mapsto G(L\hotimes \J)$ is diffeotopy invariant.
\end{them}
\begin{exe}\label{exe:diforeg}
Prove:
\item{i)} If $L$ is a locally convex algebra, then $L[t]$ is a locally convex algebra, and there is an isomorphism $L[t]\cong L\hotimes \C[t]$ where
$\C[t]$ is equipped with the fine topology.
\item{ii)} Let $G$ be a diffeotopy invariant functor from locally convex algebras to abelian groups. Prove that $G$ is polynomial homotopy
invariant.
\end{exe}

The following fact shall be needed below.
\begin{lem}\label{fact:jcomstable}
Let $G$ be a functor from locally convex algebras to abelian groups, and $\J$ a Fr\'echet ideal. Assume that $G$ is $M_2$-stable and that
$F(-):=G(-\hotimes\J)$ is diffeotopy invariant. Then $F$ is $\com$-stable.
\end{lem}
\begin{proof}
Let $\iota:\C\to \com$ be the inclusion; put $\alpha=1_\J\hotimes\iota$. We have to show that if $L$ is a locally convex algebra, then $G$ maps
$1_L\hotimes\alpha$ to an isomorphism. To do this one constructs a map
$\beta:\com\hotimes\J\to \com$, and shows that $G(1_L\hotimes\beta)$ is inverse to $G(1_L\hotimes\alpha)$. To
define $\beta$, proceed as follows. By \cite[5.1.3]{cot}, $\J\supset\cL^1$, and the tensor product of operators defines a map
$\theta:\cL^1\hotimes\J\to \J$. Write $\phi:\com\to\cL^1$ for the inclusion. Put $\beta=\theta\circ(\phi\hotimes 1_\J)$.
The argument of the proof of \cite[6.1.2]{ct} now shows that $G$ sends both $1_L\hotimes\alpha\beta$ and $1_L\hotimes\beta\alpha$
to identity maps.\qed
\end{proof}
\subsection{$KH$ of stable locally convex algebras.}

Let $L$ be a locally convex algebra. Restriction of functions defines a homomorphism of locally convex algebras $L[t]\to \cC^{\infty}([0,1],L)$, which sends $\Omega L\to \Omega^{\dif}L$. Thus we have a natural map
\begin{equation}\label{map:compakhkd}
KH_n(L)=\colim_pK_{-p}(\Omega^{p+n}L)\to \colim_pK_{-p}((\Omega^{\dif})^{p+n}L)=KD_n(L)
\end{equation}
\begin{them}\label{thm:compakh}\cite[6.2.1]{cot}
Let $L$ be a locally convex algebra, $\J$ a Fr\'echet ideal, and $A$ a $\C$-algebra. Then
\item{i)} The functors $KH_n(A\otimes_\C(-\hotimes\J))$ $(n\in\Z)$ and $K_m(A\otimes_\C(-\hotimes\J))$ $(m\le 0)$ are
diffeotopy invariant.
\item{ii)} $A\otimes_\C (L\hotimes\J)$ is $K_n$-regular $(n\le 0)$.
\item{iii)}The map $KH_n(L\hotimes \J)\to KD_n(L\hotimes \J)$ of \eqref{map:compakhkd}
 is an isomorphism for all $n$. Moreover we have
\[
 KH_n(L\hotimes \J)=KV_n(L\hotimes\J)=KV^{\dif}_n(L\hotimes J)\qquad (n\ge 1)
\]
\end{them}
\begin{proof} Part i) is immediate from \ref{thm:khppties}, \ref{exa:khkosher} and \ref{thm:hit3}. It follows from part
i) and Exercise \ref{exe:diforeg} that $A\otimes_\C(L\hotimes\J)$ is $K_n$-regular for all $n\le 0$, proving ii).
From part i) and excision, we get that the two vertical maps in the commutative diagram below are isomorphisms $(n\le 0)$:
\[
\xymatrix{K_n(L\hotimes\J)\ar[d]\ar[r]^{1}&K_n(L\hotimes\J)\ar[d]\\
          K_{n-1}(\Omega L\hotimes\J)\ar[r]&K_{n-1}(\Omega^{\dif}L\hotimes\J)}
\]
It follows that the map at the bottom is an isomorphism. This proves the first assertion of iii). The identity $KH_n(L\hotimes \J)=KV_n(L\hotimes\J)$ $(n\ge 1)$ follows from part i),
using Proposition \ref{prop:khk0reg} and Remark \ref{rem:vorst}. Similarly, part i) together with Proposition
\ref{prop:kdk0reg} imply that $KD_n(L\hotimes\J)=KV^{\dif}_n(L\hotimes\J)$ $(n\ge 1)$.\qed
\end{proof}
\begin{coro}\label{cor:kdk0k-1}
\[KH_n(A\otimes_\C (L\hotimes\J))=\begin{cases}K_0(A\otimes_\C(L\hotimes\J))& n\text{ even. }\\ K_{-1}(A\otimes_\C(L\hotimes\J))& n \text{ odd.}\end{cases}\]
\end{coro}
\begin{proof}
Put $B=A\otimes_\C(L\hotimes\J)$. By part ii) of Theorem \ref{thm:compakh} above and
Remark \ref{rem:vorst}, or directly by the proof of the theorem, we have that $B$ is $K_n$-regular for
all $n\le 0$. Thus $KH_n(B)=K_n(B)$ for $n\le 0$, by Propostion \ref{prop:khk0reg}.
To finish, we must show that $KH_n(B)$ is $2$-periodic. By \ref{fact:jcomstable}, $KH_*(A\otimes_\C(-\hotimes\J))$ is $\com$-stable.
Thus $KH_*(A\otimes_\C(\cT_0^{\sm}\hotimes\J))=0$, by Theorem \ref{thm:cusm}. Whence $KH_{*+1}(B)=KH_{*-1}(B)$, by excision
and diffeotopy invariance.\qed
\end{proof}
\begin{exa}
If $\J$ is a Fr\'echet operator ideal, then by  \ref{k0j}, \ref{k-1j} and Corollary \ref{cor:kdk0k-1}, we get:
\[
KH_n(\J)=\left\{\begin{array}{cc}\Z&n\text{ even.}\\ 0&n\text{ odd.} \end{array}\right.
\]
This formula is valid more generally for ``subharmonic" ideals (see \cite[6.5.1]{cot} for the definition of this term,
and \cite[7.2.1]{cot} for the statement). For example, the Schatten ideals $\cL_p$ are subharmonic for all $p>0$, but are Fr\'echet
only for $p\ge 1$.
\end{exa}

\section{$K$-theory spectra}\label{sec:spectra}

In this section we introduce spectra for Quillen's and other $K$-theories. For a quick introduction to spectra,
see \cite[10.9]{chubu}.

\subsection{Quillen's $K$-theory spectrum.}
Let $R$ be a unital ring.
Since the loopspace depends only on the connected component of the base point, applying the equivalence of Proposition \ref{prop:conseq} to $\Sigma R$ induces an equivalence
\begin{equation}\label{map:omegequi}
\Omega K(\Sigma R)\weq \Omega^2K(\Sigma ^2R)
\end{equation}
Moreover, by \ref{subsec:functissues}, this map is natural. Put
\[
{}_n\K R:=\Omega K(\Sigma^{n+1}R).
\]
The equivalence \eqref{map:omegequi} applied to $\Sigma^nR$ yields an equivalence
\[
{}_n\K R\weq \Omega{}(_{n+1}\K R).
\]
The sequence $\K R=\{{}_n\K R\}$ together with the homotopy equivalences above constitute a spectrum
(in the notation of \cite[10.9]{chubu}, {\it $\Omega$-spectrum} in that of \cite[Ch. 8]{swi}),
the $K$-theory spectrum; the equivalences are the {\it bonding maps} of the spectrum. The $n$-th (stable) homotopy group of $\K R$ is
\[
\pi_n\K R=\colim_p\pi_{n+p}({}_n\K R)=K_nR\qquad (n\in\Z).
\]
Because its negative homotopy groups are in general nonzero, we say that the spectrum $\K R$ is {\it nonconnective}.
Recall that the homotopy category of spectra $\ho Spt$ is triangulated, and, in particular, additive.
In the first part of the proposition below, we show that $\ass_1\to\ho Spt$, $R\mapsto \K R$ is an additive functor.
Thus we can extend the functor $\K$ to all (not necessarily unital) rings, by
\begin{equation}\label{ksnunital}
\K A:=\hofi(\K(\tilde{A})\to \K (\Z))
\end{equation}

\begin{propo}\label{prop:ksppties}
\item{i)} The functor $\K:\ass_1\to \ho Spt$ is additive.
\item{ii)} The functor $\K:\ass\to \ho Spt$ defined in \eqref{ksnunital} above, is $M_\infty$-stable
on unital rings.
\end{propo}
\begin{proof} It follows from \ref{prop:hikay} i) and iii).\qed
\end{proof}

If $A\triqui B$ is an ideal, we define the relative
$K$-theory spectrum by
\[
\K(B:A)=\hofi (\K(B)\to \K(B/A)).
\]

\begin{propo}\label{prop:exciexci}
Every short exact sequence of rings \eqref{abc} with $A$ $K$-excisive, gives rise to a distinguished triangle
\[
\K A\to \K B\to \K C\to \Omega^{-1}\K A
\]
\end{propo}
\begin{proof} Immediate from Theorem \ref{thm:exito}.\qed
\end{proof}

\subsection{$KV$-theory spaces.}
Let $A$ be a ring. Consider the simplicial ring
\[
\DA A:[n]\mapsto A\otimes\Z[t_0,\dots,t_n]/\langle1-(t_0+\dots+t_n)\rangle.
\]
It is useful to think of elements of $\DA_nA$ as formal polynomial functions on
the algebraic $n$-simplex $\{(x_0,\dots,x_n)\in \Z^{n+1}:\sum x_i=1\}$ with
values in $A$. Face and degeneracy maps are given by
\begin{gather}\label{formucaradege}
d_i(f)(t_0,\dots,t_{n-1})=f(t_0,\dots,t_{i-1},0,t_i,\dots,t_{n})\\
s_j(f)(t_0,\dots,t_{n+1})=f(t_0,\dots,t_{i-1},t_i+t_{i+1},\dots,t_{n+1}).\nonumber
\end{gather}
Here $f\in \DA_nA$, $0\le i\le n$, and $0\le j\le n-1$.\goodbreak
In the next proposition and below, we shall use the geometric realization of a simplicial space;
see \cite[I.3.2 (b)]{GM} for its definition. We shall also be concerned with simplicial groups; see \cite[Ch.8]{chubu} for a brief introduction to the
latter. The following proposition and the next are taken from D.W. Anderson's paper \cite{anderson}.

\begin{propo}(\cite[1.7]{anderson})
Let $A$ be a ring and $n\ge 1$. Then $KV_nA=\pi_{n-1}\GL\Delta A=\pi_n|B\GL\DA A|$.
\end{propo}
\begin{proof}
The second identity follows from the fact that if $G$ is a group,
then $\Omega B G\weq G$ \cite{bf} and the fact that, for a
simplicial connected space $X$, one has $\Omega |X|\weq |\Omega X|$. To prove
the first identity, proceed by induction on $n$. Write $\sim$ for
the polynomial homotopy relation in $\GL A$ and $\coeq$ for the
coequalizer of two maps. The case $n=1$ is
\begin{align*}
\pi_0\GL(\Delta A)=&\coeq(\GL\Delta_1A\overset{\ev_0}{\underset{\ev_1}{\rightrightarrows}}\GL A)\\
                        =&\GL A/\sim\\
                        =&\GL A/\GL(A)'_0=KV_1A.
\end{align*}
For the inductive step, proceed as follows. Consider the exact sequence of rings
\[
0\to \Omega A\to PA\to A\to 0
\]
Using that $\GL(-)'_0=\ima (\GL P(-)\to \GL(-))$ and that $P\DA=\DA P$ and $\Omega\DA=\DA\Omega$, we obtain exact sequences of simplicial groups

\begin{gather}
\xymatrix{1\ar[r]&\GL\DA\Omega A\ar[r]&\GL\DA PA\ar[r]&\GL(\DA A)'_0\ar[r]&1}\label{seq:decapi1}\\
\xymatrix{1\ar[r]&\GL(\DA A)'_0\ar[r]&\GL\DA A\ar[r]&KV_1(\DA A)\ar[r]&1}\label{seq:decapi2}
\end{gather}
Since $KV_1$ is homotopy invariant (by \ref{prop:kv1ppties}), we have $\pi_0KV_1\DA A=KV_1A$ and $\pi_nKV_1\DA A=0$ for $n>0$. It follows
from \eqref{seq:decapi2} that
\begin{equation}\label{pigl0}
\pi_n\GL(\DA A)'_0=\begin{cases} 0&n=0\\ \pi_n\GL(\DA A)& n\ge 1\end{cases}
\end{equation}
Next, observe that there is a split exact sequence
\[
1\to \GL \DA PA\to \GL \DA A[x]\to \GL\DA A\to 1
\]
Here, the surjective map and its splitting are respectively $\GL d_0$ and $\GL s_0$. One checks that the maps
\begin{gather*}
h_i:\DA_n\DA_1 A\to \DA_{n+1}A,\\ h_i(f)(t_0,\dots,t_n,x)=f(t_0,\dots,t_i+t_{i+1},\dots, t_n,(t_{i+1}+\dots +t_n)x)
\end{gather*}
$0\le i\le n$ form a simplicial homotopy between the identity and the map $\DA_n(s_0d_0)$.
Thus $\GL d_0$ is a homotopy equivalence, whence $\pi_*\GL\DA PA=0$. Putting this together with \eqref{pigl0} and using the homotopy exact sequence
of \eqref{seq:decapi1}, we get
\[
\pi_n\GL\DA \Omega A=\pi_{n+1}\GL\DA A \qquad(n\ge 0).
\]
The inductive step is immediate from this.\qed
\end{proof}
\begin{exe} Let $L$ be a locally convex algebra. Consider the
{\it geometric $n$-simplex}
\[
\A^n:=\{(x_0,\dots,x_n)\in\R^{n+1}:\sum x_i =1 \}\supset
\Delta^n:=\{x\in \A^n:x_i\ge 0 \ (0\le i\le n)\}.
\]
If $L$ is a locally convex algebra, we write
\[
\DD_n L:=\cC^\infty(\Delta^n,L).
\]
Here, $\cC^\infty(\Delta^n, -)$ denotes the locally convex
vectorspace of all those functions on $\Delta^n$ which are
restrictions of $\cC^\infty$-functions on $\A^n$. The cosimplicial
structure on $[n]\mapsto\Delta^n$ induces a simplicial one on $\DD
L$. In particular, $\DD L$ is a simplicial locally convex algebra,
and $\GL(\DD L)$ is a simplicial group.
\item{i)} Prove that $KV_n^{\dif}L=\pi_{n-1}\GL(\DD L)$ $(n\ge
1)$.
\item{ii)} Let $A$ be a Banach algebra. Consider the simplicial Banach algebra $\Delta_*^{\top}A=\cC(\Delta^*,A)$
and the simplcial group $\GL(\Delta^{\top}A)$. Prove that
$K^{\top}_nA=\pi_{n-1}\GL(\Delta^{\top}A)$ $(n\ge 1)$.
\end{exe}
\begin{propo}\label{prop:ander}(\cite[2.3]{anderson})
Let $R$ be a unital ring. Then the map $|B\GL \Delta R|\to |K\Delta R|$ is an equivalence.\qed
\end{propo}
\begin{coro}
If $A$ is a ring and $n\ge 1$, then
\[
KV_nA=\pi_n|K(\Delta \tilde{A}:\Delta A)|\qed
\]
\end{coro}
\begin{rem}
The argument of the proof of Proposition \ref{prop:ander} in
\cite{anderson} applies verbatim to the $\cC^\infty$ case, showing
that if $T$ is a unital locally convex algebra, then
\[
|B\GL
\DD T|\weq |K\DD T|.
\]
It follows that if $L$ is any, not necessarily unital locally convex algebra and $\tilde{L}_\C=L\oplus \C$
is its unitalization, then
\[
KV^{\dif}_nL=\pi_n|K(\DD \tilde{L}_\C:\DD L)|
\]
The analogous formulas for the topological $K$-theory of Banach algebras are also true and can be derived in the same manner.
\end{rem}
\subsection{The homotopy $K$-theory spectrum.}
Let $R$ be a unital ring. Consider the simplicial spectrum $\K\DA R$. Put
\[
\KH(R)=|\K \Delta R|
\]
One checks that $\KH:\ass_1\to\ho Spt$ is additive. Thus $\KH$ extends to arbitrary rings by
\[
\KH(A)=\hofi(\KH\tilde{A}\to\KH\Z)=|\K (\Delta \tilde{A}:\Delta A)|
\]
\begin{rem}\label{rem:kh_berreta}
If $A$ is any, not necessarily unital ring, one can also consider the spectrum $|\K \DA A|$; the map
\begin{equation}\label{map:berrekh1}
\K\DA A=\K(\widetilde{\DA A}:\DA A)\to \K (\DA \tilde{A}:\DA A)
\end{equation}
induces
\begin{equation}\label{map:berrekh2}
|\K\DA A|\to |\KH A|.
\end{equation}
If $A$ happens to be unital, then \eqref{map:berrekh1} is an equivalence, whence the same is true of \eqref{map:berrekh2}.
Further, we shall show below that \eqref{map:berrekh2} is in fact an equivalence for all
$\Q$-algebras $A$.
\end{rem}
\begin{propo}
Let $A$ be a ring, and $n\in \Z$. Then $KH_n(A)=\pi_n\KH(A)$.
\end{propo}
\begin{proof}
It is immediate from the definition of the spectrum $\KH A$ given
above that
\[\pi_*\KH(A)=\ker(\pi_*\KH\tilde{A}\to\pi_*\KH\Z)\]
Since a similar formula holds for $KH_*$, it suffices to prove the
proposition for unital rings. Let $R$ be a unital ring. By
definition, the spectrum $|\KH(R)|$ is the spectrification of the
pre-spectrum whose $p$-th space is $|\Omega K\Delta
\Sigma^{p+1}R|$. Thus
\begin{align*}
\pi_n\KH(R)=&\colim_p\pi_{n+p}|\Omega K\Delta \Sigma^{p+1}
R|=\colim_p\pi_{n+p}\Omega| K\Delta \Sigma^{p+1} R|\\
=&\colim_p\pi_{n+p+1}|K\Delta \Sigma^{p+1}
R|=\colim_pKV_{n+p}\Sigma^pR=KH_nR.\ \ \qed
\end{align*}
\end{proof}
\begin{exe}\label{exer:kd}
Let $L$ be a locally convex algebra. Put
\[
\KD L=|\K(\DD \tilde{L}_\C: \DD L)|.
\]
\item{i)} Show that $\pi_n\KD L=KD_n L$ ($n\in\Z$).
\item{ii)} Construct a natural map
\[\K\DD L\to \KD L \]
and show it is an equivalence for unital $L$.
\end{exe}
\section{Primary and secondary Chern characters}\label{sec:chern}

In this section, and for the rest of the paper, all rings considered will be $\Q$-algebras.

\subsection{Cyclic homology.}

The different
variants of cyclic homology of an algebra $A$ are related by an exact sequence, Connes' $SBI$ sequence
\begin{equation}\label{seq:sbi}
\xymatrix{HP_{n+1}A\ar[r]^S& HC_{n-1}A\ar[r]^B& HN_nA\ar[r]^I&HP_nA\ar[r]^S& HC_{n-2}A}
\end{equation}
Here $HC$, $HN$ and $HP$ are respectively cyclic, negative cyclic and periodic cyclic homology.
The sequence \eqref{seq:sbi} comes from an exact sequence of complexes of $\Q$-vectorspaces. The complex for cyclic homology
is Connes' complex $C^\lambda A$, whose definition we shall recall presently; see \cite[5.1]{lod} for the negative
cyclic and periodic cyclic complexes.
The complex $C^\lambda A$ is a nonnegatively
graded chain complex, given in dimension $n$ by the coinvariants
\begin{equation}\label{clambda}
C^\lambda_nA:=(A^{\otimes n+1})_{\Z /(n+1)\Z }
\end{equation}
of the tensor power --taken over $\Z$, or, what is the same, over $\Q$-- under the action of $\Z/(n+1)\Z$
defined by the signed cyclic permutation
\[
\lambda(a_0\otimes\dots\otimes a_n)=(-1)^na_n\otimes a_0\otimes\dots\otimes a_{n-1}.
\]
The boundary map $b:C_n^\lambda A\to C_{n-1}^\lambda A$ is induced by
\begin{multline*}
b:A^{\otimes n+1}\to A^{\otimes n},\quad b(a_0\otimes\dots\otimes a_n)=\sum_{i=0}^{n-1}(-1)^ia_0\otimes\dots\otimes a_ia_{i+1}
\otimes\dots\otimes a_n\\+(-1)^na_na_0\otimes\dots\otimes a_{n-1}
\end{multline*}
\begin{exa}\label{exa:hc0}
The map $C^\lambda_1(A)\to C_0^\lambda(A)$ sends the class of $a\otimes b$ to $[a,b]:=ab-ba$. Hence
\[
HC_0A=A/[A,A].
\]
\end{exa}

By definition,  $HC_nA=0$ if $n<0$.
Also by definition, $HP$ is periodic of period $2$.

The following theorem subsumes the main properties of $HP$.

\begin{them}\label{thm:hp_ppties}
\item{i)} (Goodwillie, \cite{goo}; see also \cite{cq2}) The functor $HP_*:\Q-\ass\to \ab$ is homotopy
invariant and nilinvariant.
\item{ii)}(Cuntz-Quillen, \cite{cq}) $HP$ satisfies excision for $\Q$-algebras; to each exact sequence
\eqref{abc} of  $\Q$-algebras, there corresponds a $6$-term exact sequence
\begin{equation}\label{seq:exciper}
\xymatrix{HP_0A\ar[r]&HP_0B\ar[r]&HP_0C\ar[d]\\ HP_1C\ar[u]&HP_1B\ar[l]&HP_1A\ar[l]}
\end{equation}
\smartqed\qed
\end{them}
\begin{rem}
The sequence \eqref{seq:sbi} comes from an exact sequence of complexes, and thus, via the Dold-Kan correspondence,
it corresponds to a homotopy fibration of spectra
\begin{equation}\label{seq:spt_sbi}
\Omega^{-1}\HC A\to\HN A\to\HP A
\end{equation}
Similarly, the excision sequence \eqref{seq:exciper} comes from a cofibration sequence in the category of
pro-supercomplexes \cite{cv}; applying the Dold-Kan functor and taking homotopy limits yields a homotopy
fibration of Bott-periodic spectra
\begin{equation}\label{seq:sptexciper}
\HP A\to\HP B\to\HP C
\end{equation}
The sequence \eqref{seq:exciper} is recovered from \eqref{seq:sptexciper} after taking homotopy groups.
\end{rem}
\subsection{Primary Chern character and infinitesimal $K$-theory.}
The main or primary character is a map going from $K$-theory to negative cyclic homology
\[
c_n:K_nA\to HN_nA\qquad (n\in\Z).
\]
(See \cite[Ch. 8, Ch. 11]{lod} for its definition).
This group homomorphism is induced by a map of spectra
\[
\K A\to \HN A
\]
Put $\K^{\inf}A:=\hofi (\K A\to \HN A)$
for its fiber; we call $K_*^{\inf}A$ the {\it infinitesimal $K$-theory} of $A$. Thus, by definition,
\begin{equation}\label{basicfib}
\K^{\inf}A\to \K A\to \HN A
\end{equation}
is a homotopy fibration. The main properties of $K^{\inf}$ are subsumed in the following theorem.
\begin{them}\label{thm:ppties_kinf}
\item{i)} (Goodwillie, \cite{goo1}) The functor $K_n^{\inf}:\Q-\ass\to \ab$ is nilinvariant $(n\in\Z)$.
\item{ii)}(\cite{kabi}) $K^{\inf}$ satisfies excision for $\Q$-algebras. Thus to every exact sequence
of $\Q$-algebras \eqref{abc} there corresponds a triangle
\[
\K^{\inf}A\to \K^{\inf}B\to \K^{\inf}C\to\Omega^{-1}\K^{\inf}A
\]
in $\ho(Spt)$ and therefore an exact sequence
\[
K^{\inf}_{n+1}C\to K^{\inf}_nA\to K^{\inf}_nB\to K^{\inf}_nC\to K^{\inf}_{n-1}A\ \ \qed
\]
\end{them}

\subsection{Secondary Chern characters.}

Starting with the fibration sequence \eqref{basicfib}, one builds up a commutative diagram with homotopy
fibration rows and columns
\begin{equation}\label{fundechar}
\xymatrix{\K^{\inf,\nil}A\ar[d]\ar[r]&\K^{\inf}A\ar[r]\ar[d]&|\K^{\inf}\DA A|\ar[d]\\
           \K^{\nil}A\ar[r]\ar[d] & \K A\ar[d]_{c}\ar[r]&|\K\DA A|\ar[d]_{c\DA}\\
           \HN^{\nil}A\ar[r]&\HN A\ar[r]&|\HN \DA A|.}
\end{equation}
The middle column is \eqref{basicfib}; that on the right is \eqref{basicfib} applied to $\DA A$; the horizontal map of
homotopy fibrations from middle to right is induced by the inclusion $A\to \DA A$, and its fiber is the column on the left.

\begin{lem}\label{lem:vanisimpli}(\cite[2.1.1]{cot})
Let $A$ be a simplicial algebra; write $\pi_*A$ for its homotopy groups. Assume
$\pi_nA=0$ for all $n$. Then $\HC A\weq 0$ and $\HN A\weq\HP A$.\qed
\end{lem}
\begin{propo}\label{prop:identiboto}
Let $A$ be a $\Q$-algebra. Then there is a weak equivalence of fibration sequences
\[
\xymatrix{\HN^{\nil}A\ar[r]\ar[d]^\wr &\HN A\ar[r]\ar[d]^\wr &\HN \DA A\ar[d]^\wr \\ \Omega^{-1}\HC A\ar[r]&\HN A\ar[r]& \HP A}
\]
\end{propo}
\begin{proof}
By Lemma \ref{lem:vanisimpli} and Theorem \ref{thm:hp_ppties}, we have equivalences
\[
\xymatrix{\HN \DA A\ar[r]^{\sim}& \HP \DA A & \HP A\ar[l]_\sim}
\]
The proposition is immediate from this.\qed
\end{proof}
\begin{propo}\label{prop:berrebuena}
If $A$ is a $\Q$-algebra, then the natural map $|\K\DA A|\to \KH A$ of \eqref{map:berrekh2} above
is an equivalence.
\end{propo}
\begin{proof}
We already know that the map is an equivalence for unital algebras. Thus since $\KH$ is excisive, it suffices
to show that $\K\DA(-)$ is excisive. Using Proposition \ref{prop:identiboto} and diagram \eqref{fundechar},
we obtain a homotopy fibration
\[
\K^{\inf}\DA A\to \K\DA A\to \HP A
\]
Note $\HP$ is excisive by Cuntz-Quillen's theorem \ref{thm:hp_ppties} ii). Moreover, $\K^{\inf}\DA(-)$
is also excisive, because $K^{\inf}$ is excisive (\ref{thm:ppties_kinf} ii)), and because $\DA(-)$ preserves exact sequences and $|-|$ preserves
fibration sequences. It follows that $\K\DA(-)$ is excisive; this completes the proof.\qed
\end{proof}

In view of Propositions \ref{prop:identiboto} and \ref{prop:berrebuena}, we may replace diagram \eqref{fundechar}
by a homotopy equivalent diagram
\begin{equation}\label{fundechar2}
\xymatrix{\K^{\inf,\nil}A\ar[d]\ar[r]&\K^{\inf}A\ar[r]\ar[d]&|\K^{\inf}\DA A|\ar[d]\\
           \K^{\nil}A\ar[r]\ar[d]_\nu & \K A\ar[d]_{c}\ar[r]&\KH A\ar[d]_{ch}\\
           \Omega^{-1}\HC A\ar[r]&\HN A\ar[r]&\HP A.}
\end{equation}
The induced maps $\nu_*:K^{\nil}A\to HC_{*-1}A$ and $ch_*:KH_*A\to HP_*A$ are the {\it secondary} and the
{\it homotopy} Chern characters. By definition, they fit together with the primary character $c_*$ into
a commutative diagram with exact rows
\begin{equation}\label{htpyfundechar}
\xymatrix{ KH_{n+1}A\ar[r]\ar[d]_{ch_{n+1}}& K^{\nil}_nA\ar[r]\ar[d]_{\nu_n}&K_n A\ar[r]\ar[d]_{c_n}& KH_nA
\ar[d]_{ch_n}\ar[r]&\tau K^{\nil}_{n-1}A\ar[d]_{\nu_{n-1}}\\
               HP_{n+1}A\ar[r]_S&HC_{n-1}A\ar[r]_B&HN_nA\ar[r]_I&HP_nA\ar[r]_S&HC_{n-2}A.}
\end{equation}
\goodbreak
\begin{rem}
The construction of secondary characters given above goes back to Weibel's paper \cite{wenil}, where a
a diagram similar to \eqref{htpyfundechar}, involving Karoubi-Villamayor $K$-theory $KV$ instead of $KH$
(which had not yet been invented by Weibel), appeared (see also \cite{karmult}).
For $K_0$-regular algebras and $n\ge 1$, the latter diagram is equivalent to \eqref{htpyfundechar}.
\end{rem}
Recall that, according to the notation of Section \ref{sec:polikv}, an algebra is $K^{\inf}_n$-regular
if $K^{\inf}_nA\to K^{\inf}_n\DA_p A$ is an isomorphism for all $p\ge 0$. We say that $A$ is
{\it $K^{\inf}$-regular} if it is $K_n^{\inf}$-regular for all $n$.

\begin{propo}\label{prop:kinfreg}
Let $A$ be a $\Q$-algebra. If $A$ is $K^{\inf}$-regular, then the secondary character
$\nu_*:K^{\nil}_*A\to HC_{*-1}A$ is an isomorphism.
\end{propo}
\begin{proof} The hypothesis implies that the map $\K^{\inf} A\to \K^{\inf}\DA_n A$ is a weak
equivalence $(n\ge 0)$.
Thus, viewing $\K^{\inf} A$ as a constant simplicial spectrum and taking realizations, we obtain an equivalence
$\K^{\inf} A\weq |\K^{\inf}\DA A|$. Hence $\K^{\inf,\nil}A\weq 0$ and therefore $\nu$ is an equivalence.\qed
\end{proof}

\begin{exa}
The notion of $K^{\inf}$-regularity of $\Q$-algebras was introduced in \cite[\S3]{cot}, where some examples are
given and some basic properties
are proved; we recall some of them. First of all, for $n\le -1$, $K_n^{\inf}$-regularity is the same
as $K_n$-regularity. A $K_0^{\inf}$-regular algebra is $K_0$-regular, but not conversely. If $R$ is
unital and $K^{\inf}_1$-regular, then the two sided ideal $<[R,R]>$ generated by the additive
commutators $[r,s]=rs-sr$ is the whole ring $R$. In particular, no nonzero unital commutative ring is
$K_1^{\inf}$-regular. Both infinite sum and nilpotent algebras are $K^{\inf}$-regular.
If \eqref{abc} is an exact sequence of $\Q$-algebras such that any two of $A$, $B$, $C$ are
$K^{\inf}$-regular, then so is the third. \goodbreak

We shall see in \ref{thm:laseq} that any stable locally convex algebra is $K^{\inf}$-regular.
\end{exa}
\subsection{Application to $KD$.}
\begin{propo} Let $L$ be a locally convex algebra. Then the natural map $\K \DD L\to \KD L$ of
\ref{exer:kd} ii) is an equivalence.
\end{propo}
\begin{proof} By Exercise \ref{exer:kd} ii), the proposition is true for unital $L$. Thus it suffices to show
that $\K\DD(-)$ satisfies excision for those exact sequences \eqref{seq:lmn} which admit a continuous linear
splitting. Applying the sequence \eqref{basicfib} to $\DD L$ and taking realizations yields a fibration sequence
\[
|\K^{\inf}\DD L|\to |\K \DD L|\to |\HN\DD L|
\]
One checks that $\pi_*\DD L=0$ (see \cite[4.1.1]{cot}). Hence the map $I:\HN\DD L\to \HP\DD L$ is an equivalence, by Lemma
\ref{lem:vanisimpli}. Now proceed as in the proof of Proposition \ref{prop:berrebuena}, taking into account
that $\DD(-)$ preserves exact sequences with continuous linear splitting.\qed
\end{proof}
\begin{coro}\label{cor:regdif}
Assume that the map $K_nL\to K_n\DD_pL$ is an isomorphism for all $n\in \Z$ and all $p\ge 0$.
Then $\K L\to \KD L$ is an equivalence.
\end{coro}
\begin{proof}
Analogous to the first part of the proof of Proposition \ref{prop:kinfreg}.\qed
\end{proof}
\section{Comparison between algebraic and topological $K$-theory III}\label{sec:karconfre}
\subsection{Stable Fr\'echet algebras.}

The following is the general version of theorem \ref{thm:karconbau}, also due to Wodzicki.
\begin{them}\label{thm:karconfre}(\cite[Thm. 2]{wod},\cite[8.3.3, 8.3.4]{cot})
Let $L$ be an $m$-Fr\'echet algebra with uniformly bounded left or right approximate unit.
Then there is a natural isomorphism:
\[ K_n(L \hotimes \cK) \stackrel{\sim}{\to} KD_n(L \hotimes \cK), \quad \forall n \in \Z.\]
\end{them}
\begin{proof} Write $\ca$ for the full subcategory of those locally convex algebras which are $m$-Fr\'echet
algebras with left ubau. In view of Corollary \ref{cor:regdif}, it suffices to show that for all $n\in\Z$ and $p\ge 0$,
the map
\begin{equation}\label{toprove}
K_n(L\hotimes \cK)\to K_n(\DD_p L\hotimes\cK)
\end{equation}
is an isomorphism for each $L\in\ca$. Note that, since $\DD_p\C$ is a unital $m$-Fr\'echet algebra and its unit is uniformly
bounded, the functor $\DD_p(-)=-\hotimes\DD_p\C$ maps $\ca$ into itself. Since $L\to \DD_pL$ is a diffeotopy
equivalence,
this means that to prove \eqref{toprove} is to prove that $K_n(-\hotimes\cK):\ca\to \ab$ is diffeotopy
invariant. Applying the same argument as in the proof of Theorem \ref{thm:karconbau}, we get that the natural
map
\[
K_*(L\hotimes\cK)\to K_*((L\hotimes\cK)[0,1])
\]
is an isomorphism. It follows that $K_*(-\hotimes\cK)$ is invariant under continous homotopies, and thus also under diffeotopies.
\end{proof}
\begin{exe} Prove that if $L$ is as in Theorem \ref{thm:karconfre} and $M=L\hotimes\cK$, then
$KD_*(M(0,1))=KD_{*+1}M$.
\end{exe}
\begin{exe}\label{exe:summ3}
Prove that the map $K_n(L \hotimes \cK) {\to} KD_n(L \hotimes \cK)$ is an isomorphism for every unital
Fr\'echet algebra $L$, even if the unit of $L$ is not uniformly bounded. (Hint: use Exercise \ref{exe:summ2}).
\end{exe}
\begin{rem}\label{rem:ncp}
N.C. Phillips has defined a $K^{\top}$ for $m$-Fr\'echet algebras (\cite{ncp}) which extends that of Banach algebras
discussed in Section \ref{sec:topk} above. We shall see presently that, for $L$ as in Theorem \ref{thm:karconfre},
\[
K_*^{\top}(L\hotimes \cK)=KD_*(L\hotimes\cK)=K_*(L\hotimes\cK).
\]
Phillips' theory is Bott periodic and satisfies $K^{\top}_0(M)=K_0(M\hotimes\com)$ and
$K_1^{\top}(M)=K_0((M\hotimes\com)(0,1))$ for every Fr\'echet algebra $M$. On the other hand,
for $L$ as in the theorem, we have $KD_0(L\hotimes \cK)=K_0(L\hotimes\cK)$ and
$KD_1(L\hotimes \cK)=K_0((L\hotimes\cK)(0,1))$. But by \ref{fact:jcomstable},
$K_0(M\hotimes\cK)=K_0(M\hotimes\cK\hotimes\com)$ for every locally convex algebra $M$. This proves
that $KD_n(L\hotimes\cK)=K^{\top}_n(L\hotimes\cK)$ for $n=0,1$; by Bott periodicity, we get the equality
for all $n$.
\end{rem}

\subsection{Stable locally convex algebras: the comparison sequence.}

\begin{them}\label{thm:laseq}(see \cite[6.3.1]{cot})
Let $A$ be a $\C$-algebra, $L$ be a locally convex algebra, and $\J$ a Fr\'echet operator ideal. Then
\item{i)} $A\otimes_\C(L\hotimes\J)$ is $K^{\inf}$-regular.
\item{ii)} For each $n\in\Z$, there is a $6$-term exact
sequence
\begin{equation}\label{seq:seqtop}
\xymatrix{ K_{-1}(A\otimes_\C (L\hotimes \J))\ar[r]&HC_{2n-1}(A\otimes_\C(L\hotimes \J)\ar[r]&K_{2n}(A\otimes_\C(L\hotimes \J))\ar[d] \\
K_{2n-1}(A\otimes_\C(L \hotimes \J)) \ar[u]& HC_{2n-2}(A\otimes_\C(L \hotimes \J)) \ar[l] & K_0(A\otimes_\C(L\hotimes \J)). \ar[l]}
\end{equation}
\end{them}
\begin{proof} According to Theorem \ref{thm:ppties_kinf}, $K^{\inf}$ is nilinvariant and satisfies
excision. Hence, by Theorem \ref{thm:hit3}, $L\mapsto K_*^{\inf}(A\otimes_\C(L\hotimes\J))$ is diffeotopy invariant, whence it is invariant
under polynomial homotopies. This proves (i). Put $B=A\otimes_\C(L\hotimes \J)$. By (i) and \ref{prop:kinfreg},
$\nu_*:K^{\nil}_*B\to HC_{*-1}B$ is an isomorphism.
Hence from \eqref{htpyfundechar} we get a long exact sequence
\begin{equation}\label{seq:seqtop2}
\xymatrix{KH_{m+1}B\ar[r]& HC_{m-1}B\ar[r]&K_mB\ar[r]& KH_mB
\ar[r]^{S ch_m}&HC_{m-2}B}
\end{equation}
By Corollary \ref{cor:kdk0k-1}, $KH_{2n}B=K_0B$ and $KH_{2n-1}B=K_{-1}B$; the sequence of the theorem
follows from this, using the sequence \eqref{seq:seqtop2}.\qed
\end{proof}
\begin{coro}\label{cor:seqtop} For each $n\in\Z$, there is a $6$-term exact sequence
\begin{equation}\label{seq:seqtop}
\xymatrix{ KD_{1}(L\hotimes \J)\ar[r]&HC_{2n-1}(L\hotimes \J)\ar[r]&K_{2n}(L\hotimes \J)\ar[d] \\
K_{2n-1}(L \hotimes \J) \ar[u]& HC_{2n-2}(L \hotimes \J) \ar[l] & KD_0(L\hotimes \J). \ar[l]}
\end{equation}
\end{coro}
\begin{proof} By Theorem \ref{thm:compakh} iii), $KD_*(L\hotimes\J)=KH_*(L\hotimes\J)$. Now use
Corollary \ref{cor:kdk0k-1}.\qed
\end{proof}

\begin{exa}
We saw in Theorem \ref{thm:karconfre} that the comparison map $K_*(L\hotimes\cK)\to KD_*(L\hotimes\cK)$
is an isomorphism whenever $L$ is an $m$-Fr\'echet algebra with left ubau. Thus
\begin{equation}\label{hcvanifre}
HC_*(L\hotimes\cK)=0
\end{equation}
by Corollary \ref{cor:seqtop}. It is also possible to prove \eqref{hcvanifre} directly and deduce
Theorem \ref{thm:karconfre} from the corollary above; see \cite[8.3.3]{cot}.
\end{exa}
\begin{exa}
If we set $L=\C$ in Theorem \ref{thm:laseq} above, we obtain an exact sequence
\begin{equation}\label{seq:seqalg}
\xymatrix{ K_{-1}(A\otimes_\C  \J)\ar[r]&HC_{2n-1}(A\otimes_\C \J)\ar[r]&K_{2n}(A\otimes_\C \J)\ar[d] \\
K_{2n-1}(A\otimes_\C \J) \ar[u]& HC_{2n-2}(A\otimes_\C \J) \ar[l] & K_0(A\otimes_\C \J). \ar[l]}
\end{equation}
Further specializing to $A=\C$ and using \ref{k0j} and \ref{k-1j} yields
\[
0\to HC_{2n-1}\I\to K_{2n}\I\to \Z\overset{\alpha_n}\to HC_{2n-2}\I\to K_{2n-1}\I\to 0.
\]
Here we have written $\alpha_n$ for the composite of $S\circ ch_{2n}$ with the isomorphism $\Z\cong K_0\J$.
If for example $\J\subset\cL_p$ ($p\ge 1$) then $\alpha_n$ is injective for $n\ge (p+1)/2$, by a result of Connes and Karoubi \cite[4.13]{ck}
(see also \cite[7.2.1]{cot}). Setting $p=1$ we obtain, for each $n\ge 1$, an isomorphism
\[
K_{2n}\cL_1=HC_{2n-1}\cL_1
\]
and an exact sequence
\[
0\to \Z\overset{\alpha_n}\to HC_{2n-2}\cL_1\to K_{2n-1}\cL_1\to 0.
\]
Note that since $HC_{2n-2}\cL_1$ is a $\Q$-vectorspace by definition, the sequence above implies that $K_{2n-1}\cL_1$
is isomorphic to the sum of a copy of $\Q/\Z$ plus a $\Q$-vectorspace.
\end{exa}
\begin{rem} The exact sequence \eqref{seq:seqalg} is valid more generally for subharmonic ideals (see \cite[6.5.1]{cot} for the definition of this term,
and \cite[7.1.1]{cot} for the statement). In particular, \eqref{seq:seqalg} is valid for
all the Schatten ideals $\cL_p$, $p>0$. In \cite[7.1.1]{cot} there is also a variant of \eqref{seq:seqalg}
involving relative $K$-theory and relative cyclic homology; the particular case of the latter when $A$
is $K$-excisive is due to Wodzicki (\cite[Theorem 5]{wod}).
\end{rem}

\bibliographystyle{plain}

\printindex
\end{document}